\documentclass[10.9pt,textwidth=125mm,textheight=195mm]{article}

\usepackage{amsmath, amssymb, amsfonts, amsthm}
\usepackage{graphicx, overpic}
\usepackage{color}
\usepackage{enumerate}
\usepackage{multicol}
\usepackage{mathrsfs}
\usepackage{xcolor}
\usepackage{mathabx}
\usepackage{hyperref}

\usepackage[margin=1.25in]{geometry}

\numberwithin{equation}{section}
\numberwithin{figure}{section}

\theoremstyle{plain}
\newtheorem{thm}{Theorem}[section]
\newtheorem{lemma}[thm]{Lemma}

\newtheorem{prop}[thm]{Proposition}
\newtheorem{corollary}[thm]{Corollary}
\theoremstyle{definition}
\newtheorem{defn}[thm]{Definition}
\newtheorem{example}[thm]{Example}
\newtheorem{remark}[thm]{Remark}

\newtheorem{claim}[thm]{Claim}

\newtheorem{notation}[thm]{Notation}

\newcommand{\Z}{\mathbb{Z}}
\newcommand{\R}{\mathbb{R}}
\newcommand{\N}{\mathbb{N}}

\newcommand{\Hy}{\mathbb{H}}

\newcommand{\cC}{\mathcal{C}}

\newcommand{\cO}{\mathcal{O}}
\newcommand{\cP}{\mathcal{P}}

\newcommand{\cY}{\mathcal{Y}}

\newcommand{\diam}{\operatorname{diam}}

\newcommand{\la}{\langle}
\newcommand{\ra}{\rangle}
\newcommand{\p}{\partial}

\def\<{{\langle}}
\def\>{{\rangle}}

\definecolor{amethyst}{rgb}{0.6, 0.4, 0.8}

\newcommand{\hide}[1]{}

\title{Free products from spinning and rotating families}
\author{Mladen Bestvina\thanks{bestvina@math.utah.edu} \\ University of Utah\thanks{Department of Mathematics, University of Utah, 155 S 1400 E, Room 233, Salt Lake City, UT 84112, USA}
  \and Ryan Dickmann\thanks{dickmann@math.utah.edu} \\ University of Utah
  \and George Domat\thanks{domat@math.utah.edu, Corresponding Author} \\ University of Utah
  \and Sanghoon Kwak\thanks{kwak@math.utah.edu} \\ University of Utah
  \and Priyam Patel\thanks{patelp@math.utah.edu} \\ University of Utah
  \and Emily Stark\thanks{estark@wesleyan.edu} \\ Wesleyan University\thanks{Department of Mathematics, Wesleyan University, 265 Church Street, Middletown, CT 06459, USA}}
\date{\today}

\begin{document}

\maketitle

\begin{abstract}
  The far-reaching work of Dahmani--Guirardel--Osin~\cite{DGO} and
  recent work of Clay--Mangahas--Margalit~\cite{claymangahasmargalit}
  provide geometric approaches to the study of the normal closure of a
  subgroup (or a collection of subgroups) in an ambient group
  $G$. Their work gives conditions under which the normal closure
  in $G$ is a free product. In this paper we unify their results and
  simplify and significantly shorten the proof of the~\cite{DGO} theorem.
\end{abstract}


\section{Introduction}

Using geometry to understand the algebraic properties of a group is
a primary aim of geometric group theory. This paper focuses on
detecting when a group has the structure of a free product. The
following theorem follows from Bass-Serre theory.

\begin{thm}\label{1.1}
  Suppose a group $G$ acts on a simplicial tree $T$ without
  inversions and with trivial edge stabilizers. Suppose also that $G$ is
  generated by the vertex stabilizers $G_v$. Then, there is a subset
  $\mathcal O$ of the set of vertices of $T$ intersecting each $G$-orbit in
  one vertex such that
  $$G=*_{v\in\mathcal O} G_v.$$
\end{thm}

Dahmani--Guirardel--Osin~\cite{DGO}, inspired by the ideas of
Gromov~\cite{gromov01}, provided a far-reaching generalization of
the theorem above.  The simplicial tree above is replaced by a
$\delta$-hyperbolic space, and the group acts via a {\it very
  rotating} family of subgroups. Under these conditions, they
conclude that the group is a free product of
subgroups in the family.

\begin{thm} \cite[Theorem 5.3a]{DGO}\label{DGO}
  Let $G$ be a group acting by isometries on a
  $\delta$-hyperbolic geodesic metric space $X$, and let $\cC =
  \left(C, \{G_c \vert c \in C\}\right)$ be a $\rho$-separated
  very rotating family for a sufficiently large
  $\rho=\rho(\delta)$. Then the subgroup of $G$ generated by the
  set $\cup_{c \in C} G_c$ is isomorphic to the free product
  $\Asterisk_{c \in C'} G_c$, for some subset $C' \subset
  C$. Moreover, every element in this subgroup is either a loxodromic isometry of $X$ or it is contained in some $G_c$.
\end{thm}

The set of {\it apices} $C\subset X$ (and also the pair $\cC$) is {\it
  $\rho$-separated} if $d(c,c')\geq \rho$ for all distinct $c,c'\in
C$. The family $\{G_c\}_{c\in C}$ of subgroups of $G$ is {\it
  rotating} if
\begin{enumerate}[(i)]
  \item $C$ is $G$-invariant,
  \item $G_c$ fixes $c$, for every $c\in C$,
  \item $G_{g(c)}=gG_cg^{-1}$ for every $g\in G$ and $c\in C$.
\end{enumerate}
Note that $G_c$ is a normal subgroup of the
stabilizer $Stab_G(c)$ and similarly the subgroup $\<G_c\mid c\in
C\>$ of $G$ generated by all of the $G_c$ is normal in $G$. A rotating
family is {\it very rotating} if in addition
\begin{enumerate}[(iv)]
  \item for any distinct $c,c'\in C$ and every $g\in
        G_c\setminus\{1\}$ every geodesic between $c'$ and $g(c')$
        passes through $c$.
\end{enumerate}

Note that (iv) is a bit weaker in
the presence of sufficient separation than the definition in
\cite{DGO}; see \cite[Lemma 5.5]{DGO}. As an application,
Dahmani--Guirardel--Osin solve a long-standing open problem by
showing that the normal closure of a suitable power of any
pseudo-Anosov mapping class in a mapping class group is free and
all nontrivial elements in the normal closure are pseudo-Anosov. We discuss this in more
detail below.

An important variation of the Dahmani--Guirardel--Osin theorem was
recently proved by
Clay--Mangahas--Margalit~\cite{claymangahasmargalit}. In that
setting, the group $G$ acts on a {\it projection complex} via a
{\it spinning family} of subgroups. As an application, they
determine the isomorphism type of the normal closure of a suitable
power of various kinds of elements in the mapping class group. We
will discuss this in more detail below as well. See related work in
\cite{dahmani, dahmanihagensisto, claymangahas}.

\begin{thm} \cite[Theorem 1.6]{claymangahasmargalit}. \label{thm:cmm_intro}
  Let $G$ be a group acting by isometries on a projection
  complex~$\cP$ with vertex set $V\cP$ and preserving the
  projection data $(\mathcal Y, \{\pi_X(Y)\},\theta)$. Let $\{G_c\}_{c \in V\cP}$ be an
  $L$-spinning family of subgroups of $G$ for $L=L(\cP)$
  sufficiently large. Then the subgroup of $G$ generated by the set
  $\{G_c\}_{c \in V\cP}$ is isomorphic to the free product
  $\Asterisk_{c \in \cO} G_c$ for some subset $\cO \subset
  V\cP$. Moreover, every element of the subgroup is either
  loxodromic in $\cP$ or is contained in some $G_c$. 
\end{thm}

We next explain the terminology in the above theorem. The {\it
  projection data} is a collection of metric spaces  $\mathcal
Y=\{X,Y,Z,\cdots\}$ (with
infinite distance within a metric space allowed) together with ``projections''
$\pi_X(Y)\subset X$ for $X,Y\in\mathcal Y$ distinct, satisfying
the following {\it projection axioms} for some $\theta\geq 0$
(called the {\it projection constant},
where we set $d_X(Y,Z)=\diam (\pi_X(Y)\cup \pi_X(Z))$:
\begin{enumerate}[(P1)]
  \item $\diam \pi_X(Y)\leq\theta$, for any $X\neq Y$,
  \item (the Behrstock inequality) if $d_X(Y,Z)>\theta$ then
        $d_Y(X,Z)\leq\theta$, and
  \item for any $X,Y$ the set $\{Z\neq X,Y\mid
        d_Z(X,Y)>\theta\}$ is finite.
\end{enumerate}
From this data, Bestvina--Bromberg--Fujiwara
\cite{bestvinabrombergfujiwara} construct a graph
$\cP=\cP(\mathcal Y)$, called the {\it projection complex}, with
the vertices in 1-1 correspondence with the spaces in $\mathcal
Y$. 
Roughly speaking, $X$ and $Y$ are connected by an edge if  $d_Z(X,Y)$ is small for any $Z$. This graph is connected, and it is
quasi-isometric to a tree. Any group $G$ acting by isometries on
the disjoint union $\bigsqcup_{X\in\mathcal Y} X$, permuting the
spaces and commuting with projections
(i.e. $g(\pi_X(Y))=\pi_{g(X)}g(Y)$), acts by isometries on
$\cP$, and we say that $G$ {\it preserves the projection data}.

An {\it $L$-spinning family} is a family $\{G_c\}$ parametrized by
the vertices $c\in V\cP$ satisfying
\begin{enumerate}[(a)]
  \item $G_c$ fixes $c$,
  \item $G_{g(c)}=gG_cg^{-1}$ for $g\in G$, $c\in C$, and
  \item $d_c(c',g(c'))>L$ for $c\neq c'$ and $g\in
        G_c\smallsetminus\{1\}$.
\end{enumerate}

The main goal of this paper is to simplify and significantly
shorten the proof of the Dahmani--Guirardel--Osin theorem using
the Clay--Mangahas--Margalit theorem and the machinery of
projection complexes. We also present a variant of the proof of
the~\cite{claymangahasmargalit} theorem to directly construct an
action of the group on a tree as in Theorem 1.1. Given a group
action on a $\delta$-hyperbolic metric space equipped with a very rotating
family of subgroups, we construct an action of that group on a
projection complex with the same family acting as a spinning
family. While our proof of Theorem~\ref{thm:cmm_intro} still uses
the construction of windmills (which are used in \cite{DGO} and
\cite{gromov01}), our work differs from
\cite{claymangahasmargalit} in that we find a natural tree on
which $G$ acts as in Theorem \ref{1.1} and eliminate the need to
work with normal forms. We also introduce the notion of {\it
  canoeing} in a projection complex, which is inspired by the
classic notion of canoeing in the hyperbolic plane (see
Section~\ref{sec:canoe}), and enables us to further streamline
some of the arguments from~\cite{claymangahasmargalit}. 

\begin{thm}\label{thm:intro_build_complex}
  Let $G$ be a group acting by isometries on a $\delta$-hyperbolic
  metric space $X$. Let $\cC = \bigl(C, \{G_c \, | \, c \in C\}\bigr)$ be a rotating family, where $C \subset X$ is $\rho$-separated for $\rho \geq 20\delta$ and $G_c \leq G$. Then the following hold.
  \begin{enumerate}[(1)]
    \item The group $G$ acts by isometries and preserves the
          projection data on a projection complex associated to $C$,
          with the projection constant $\theta=\theta(\delta)$.
    \item If $\cC$ is a very rotating family, then the
          family of subgroups $\{G_c\}_{c \in \cC}$ forms an
          $L(\rho)$-spinning family for the action of $G$ on the
          projection complex.
    \item $L(\rho)\to\infty$ as $\rho\to\infty$. In particular, we can take $L = 2^{\frac{\rho-6\delta-4}{2\delta}} - 4 - 248\delta$, so that $L$ grows exponentially with respect to $\rho$. 
  \end{enumerate}
\end{thm}

To prove Theorem~\ref{thm:intro_build_complex}, we construct a
projection complex via the Bestvina--Bromberg--Fujiwara axioms. These
axioms require us to first define for each $c \in C$ a metric space
$S_c$ (with infinite distances allowed) and projections
$\pi_{S_c}(S_{c'})$, which we abbreviate to $\pi_c(c')$. A standard example of
such a construction is the following. Take a closed hyperbolic surface
$S$ and a closed (not necessarily simple) geodesic $\alpha$. Consider
the universal cover $\tilde S=\mathbb H^2$ and the set $\mathcal Y$ of all
lifts of $\alpha$. For two different lifts $A,B\in\mathcal Y$, define
$\pi_A(B)\subset A$ to be the nearest point projection of $B$ to
$A$. This will be an open interval in $A$ whose diameter is uniformly
bounded independently of $A,B$ (but which depends on $\alpha$). Roughly
speaking, $\pi_A(B)$ can have a large diameter only of $B$ fellow
travels $A$ for a long time. It is not hard to see that the projection
axioms hold in this case. A similar construction can be carried out
when $S$ is a hyperbolic surface with a cusp and $\alpha$ is a
horocycle. Now $\mathcal Y$ is an orbit of pairwise disjoint horocycles in
$\mathbb H^2$ and $\pi_A(B)$ is defined as the nearest point
projection of $B$ to $A$ as before. There are now two natural choices
of a metric on horocycles in $\mathcal Y$: one can take the intrinsic metric
so that it is isometric to $\R$ or the induced metric from $\mathbb
H^2$. Either choice satisfies the projection axioms, but note here that
the intrinsic metric can also be defined as the path metric where
paths are not allowed to intersect the open horoball cut out by the
horocycle.

The starting point of our proof of Theorem \ref{DGO} is the
construction of the projection complex whose vertex set is the set
$C$ of apices, inspired by the horocycle example. To each
$c\in C$ we associate the sphere $S_c$ of radius $R$ centered at
$c$. The number $R$ is chosen carefully. The open balls $B_c$ cut out
by the spheres should be pairwise disjoint and a reasonable distance
apart (a fixed multiple of $\delta$), yet big enough so that paths
in the complement of $B_c$ joining points $x,y\in S_c$ on opposite
sides of the ball are much longer (exponential in $R$) than a geodesic in $X$
joining $x,y$ (which is linear in $R$). The projection $\pi_c(c')$ for
$c,c'\in \cC$, $c\neq c'$ is the set of all points in $S_c$ that lie
on a geodesic between $c$ and $c'$. The metric we take on $S_c$ is
induced by the path metric in $X\smallsetminus B_c$ (this can take
value $\infty$ if the ball disconnects $X$). We check that with these
definitions the projection axioms hold (see Section \ref{subsec:axioms}). Thus, the group $G$
acts on the projection complex and we check that the groups $G_c$ form
a spinning family (see Section \ref{subsec:spinning}), which proves Theorem
\ref{thm:intro_build_complex}.

The same proof goes through with a slightly weaker hypothesis that the
family $\{G_c\}$ is \emph{fairly rotating}, instead of very rotating (with
slightly different constants).
Here we require only that geodesics between $c'$ and $g(c')$ pass
within $1$ of $c$,
for $g\in G_c\setminus\{1\}$, instead of
passing through $c$. This situation naturally occurs. As an
example, consider a closed hyperbolic orbifold $S$ with one cone point, with cone angle $2\pi/n$ for $n>2$. The orbifold universal cover
$\tilde S=\mathbb H^2$ admits an action by the orbifold fundamental group
$G$ where the stabilizers $G_c\cong\Z/n\Z$ of the lifts $c\in C$ of the cone
point form a rotating family. This family will never be very rotating,
but it will be fairly rotating if the pairwise distance between
distinct elements of $C$ is large enough, given by a function of
$n$. 

Note that in Theorem
\ref{thm:cmm_intro}, the constant $L(\cP)$ really depends only on the
projection constant $\theta$ and can be taken to be a fixed
multiple of $\theta$ (e.g. $1000\theta$ will do). In Theorem
\ref{thm:intro_build_complex} the projection constant $\theta$ in (1) can be
taken to be a fixed multiple of $\delta$, and $L(\rho)$ in (2) will be
an exponential function in $\rho$. Since exponential functions grow
faster than linear functions, the spinning constant $L$ in Theorem
\ref{thm:intro_build_complex} will beat the one in Theorem
\ref{thm:cmm_intro} if $\rho$ is big enough, so Theorem \ref{DGO} will
follow (see Section \ref{sec:DGO_proof}).


We now say a few words about our proof of Theorem
\ref{thm:cmm_intro}. As in~\cite{claymangahasmargalit}, we recursively
define a sequence of windmills which correspond to certain orbits of
larger and larger collections of the vertex subgroups $\{G_{c}\}$. At
each stage we prove that the these windmills have a tree-like
structure (technically, the skeleton of the canonical cover of each
windmill is a tree).  At each step we obtain a new group that is the
free product of the previous group with a suitable collection of
$G_c$'s, and taking the limit proves the theorem, see Section
\ref{subsec:cmmproof}. Canoeing enters when we verify that windmills
have a tree-like structure. The simplest example of a canoeing path in
$\cP$ would be an edge-path passing through vertices
$v_1,v_2,\cdots,v_k$ such that for every $i=2,3,\cdots,k-1$ the
``angle'' $d_{v_i}(v_{i-1},v_{i+1})$ is large. The basic properties of
projection complexes quickly imply that such paths are embedded, and
they provide a local-to-global principle enabling us to establish the
tree-like structure.

We end this introduction with some applications of Theorems \ref{DGO}
and \ref{thm:cmm_intro}. Suppose $G$ acts by
isometries on a hyperbolic space $Y$ and $g\in G$ is
loxodromic. Suppose also that $g$ is a ``WPD element'' as per Bestvina-Fujiwara
\cite{BF2002}. This amounts to saying that $g$ is contained in a
unique maximal virtually cyclic subgroup $EC(g)$ (the elementary closure of
$\la g\ra$) and further that the set of $G$-translates of a fixed
$EC(g)$-orbit is ``geometrically separated'', i.e., the nearest point
projections satisfy the projection axioms. This situation generalizes
the example above, where $Y=\mathbb H^2$, $G$ is the deck group of the
universal cover $Y\to S$, and $g$ corresponds to an indivisible element
in $G$, so $EC(g)=\la g\ra$. This situation is fairly common. For example,
$Y$ could be the curve complex of a surface of finite type, $G$ its
mapping class group, and $g\in G$ a pseudo-Anosov mapping class (see
\cite{BF2002}). For another example, take $G$ to be the Cremona group
(of birational transformations of $\mathbb C\mathbb P^2$) acting on infinite dimensional real hyperbolic space, see
\cite{CL2013}. In these situations one can construct a space $X$ by
coning off the orbit of $EC(g)$ and each translate. If the radius of
the cone is large enough, Dahmani--Guirardel--Osin show that $X$ is
hyperbolic and if $N\unlhd EC(g)$ is a sufficiently deep finite index
normal subgroup, then the set of $G$-conjugates of $N$ forms a very
rotating family with the cone points as apices. In particular, they
resolved a long-standing open problem by showing that if $g$ is a
pseudo-Anosov homeomorphism of a finite-type surface, then there is
$n>0$ such that the normal closure of $g^n$ in the mapping class group
is the free group $F_\infty$ (of infinite rank), and all nontrivial
elements of the group are pseudo-Anosov. 

Clay--Mangahas--Margalit reproved this application to mapping class groups directly from
Theorem \ref{thm:cmm_intro} and gave new applications of their
theorem. To illustrate, consider a mapping class $g$ on a finite-type
surface $S$ which is supported on a proper, connected,
$\pi_1$-injective subsurface $A\subset S$ such that $g|A$ is a
pseudo-Anosov homeomorphism of $A$. Assume also that any two
subsurfaces 
in the orbit of $A$ either coincide or intersect. There is a natural
projection complex one can construct from this setup. The vertices are the
subsurfaces in the mapping class group orbit of $A$, and the
projection $\pi_A(B)$ is the Masur-Minsky subsurface projection
\cite{MM1999,MM2000}. It follows from the work of Masur-Minsky and Behrstock
that the projection axioms hold in this setting. Clay--Mangahas--Margalit
prove that for a suitable $n>0$ the collection of conjugates of
$\la g^n\ra$ forms a spinning family and conclude that, here too, the
normal closure of $g^n$ is free. They consider more general
situations where the normal closure can be a non-free group as well. Remarkably they can exactly determine the normal closure even in this case. For example, if $S$
is a closed surface of even genus and $g$ is pseudo-Anosov supported
on exatly half the surface, then for suitable $n>0$ the normal closure
of $g^n$ is the infinite free product of copies of $F_\infty\times F_\infty$.

\subsection*{Outline} 

Preliminaries are given in Section~\ref{sec:prelims}. In Section~\ref{sec:build_proj} we construct a group action on a projection complex from the rotating family assumptions of Dahmani--Guirardel--Osin. Section~\ref{sec:canoe} contains the new proof of the result of Clay--Mangahas--Margalit via canoeing paths in a projection complex. In Section~\ref{sec:DGO_proof} we give the new proof of the result of Dahmani--Guirardel--Osin using projection complexes. Section 6 contains proofs of the moreover statements of the Dahmani--Guirardel--Osin and Clay--Mangahas--Margalit theorems. That is, we prove that elements of the corresponding groups act either loxodromically or are contained in one of the given rotating/spinning subgroups.

\subsection*{Acknowledgments} MB was partially supported by NSF DMS-1905720. GD was partially supported by NSF DMS-1607236, NSF DMS-1840190, and NSF DMS-1246989. PP was partially supported on NSF DMS-1937969. ES was partially supported by NSF DMS-1840190. We also thank the referee for numerous helpful comments.

\section{Preliminaries} \label{sec:prelims}

In this section, we state the relevant result of Dahmani--Guirardel--Osin, give background on projection complexes, state the result of Clay--Mangahas--Margalit, and give the necessary background on $\delta$-hyperbolic spaces, in that order. 

\subsection{Rotating subgroups and the result of Dahmani--Guirardel--Osin}

\begin{defn}[{\cite[Definition 2.12]{DGO}}] \label{def:very_rot} (Gromov's rotating families.) Let $G$ be a group acting by isometries on a metric space~$X$. A {\it rotating family} $\cC = (C, \{G_c \, | \, c \in C \})$ consists of a subset $C \subset X$ and a collection $\{G_c \, | \, c \in C\}$ of subgroups of $G$ such that the following conditions hold.
  \begin{enumerate}
    \item[(a-1)] The subset $C$ is $G$-invariant;
    \item[(a-2)] each group $G_c$ fixes $c$;
    \item[(a-3)] $G_{gc} = gG_cg^{-1}$ for all $g \in G$ and for all $c \in C$.
  \end{enumerate}
  The elements of the set $C$ is called the {\it apices} of the family, and the groups $G_c$ are called the {\it rotation subgroups} of the family.
  \begin{enumerate}
    \item[(b)] (Separation.) The subset $C$ is {\it $\rho$-separated} if any two distinct apices are at distance at least~$\rho$.
    \item[(c)] (Very rotating condition.) When $X$ is
          $\delta$-hyperbolic with $\delta>0$, one says that $\cC$ is {\it
          very rotating} if for all $c \in C$, all $g \in G_c - \{1\}$,
          and all $x,y \in X$ with both $d(x,c)$ and $d(y,c)$ in the
          interval $[20\delta, 40\delta]$ and $d(gx, y) \leq 15\delta$,
          then any geodesic from $x$ to $y$ contains $c$.
  \end{enumerate}   
  We will actually make use of a weaker version of the very rotating condition.
  \begin{enumerate}
    \item[(c$'$)] (Fairly rotating condition.) When $X$ is
          $\delta$-hyperbolic with $\delta>0$, one says that $\cC$ is
          \emph{fairly rotating} if for all $c \in C$, all $g \in G_{c}
          -\{1\}$, and all $x \in C$ with $x\neq c$, there
          exists a geodesic from $x$ to $gx$ that nontrivially intersects
          the ball of radius $1$ around $c$.
  \end{enumerate}
\end{defn}

\begin{remark}
  Property (c) implies Property (c$'$) by \cite[Lemma 5.5]{DGO}.
\end{remark}

\begin{example}[{\cite[Example 2.13]{DGO}}]
  Let $G = H*K$, and let $X$ be the Bass-Serre tree for this free product decomposition. Let $C \subset X$ be the set of vertices, and let $G_c$ be the stabilizer of $c \in C$. Then, $\cC = (C, \{G_c \, | \, c \in C\})$ is a $1$-separated very rotating family. 
\end{example}

Dahmani--Guirardel--Osin~\cite{DGO} prove a partial converse to the example above as follows.

\begin{thm}[{\cite[Theorem 5.3a]{DGO}}] \label{thm:DGO}
  Let $G$ be a group acting by isometries on a $\delta$-hyperbolic geodesic metric space, and let $\cC = (C, \{G_c \, | \, c \in C\})$ be a $\rho$-separated very rotating family for some $\rho \geq 200 \delta$. Then, the normal closure in $G$ of the set $\{G_c\}_{c \in C}$ is isomorphic to a free product $\Asterisk_{c \in C'} G_c$, for some (usually infinite) subset $C' \subset C$. 
\end{thm}

\subsection{Projection complexes}

Bestvina--Bromberg--Fujiwara~\cite{bestvinabrombergfujiwara} defined projection complexes via a set of projection axioms given as follows.

\begin{defn}[{\cite[Sections 1 \& 3.1]{bestvinabrombergfujiwara}, Projection axioms}] \label{def:proj_axioms}
  Let $\cY$ be a set of metric spaces (in which infinite distances are allowed), and for each $Y \in \cY$, let \[\pi_Y: \bigl(\cY - \{Y\}\bigr)  \longrightarrow  2^{Y} \] satisfy the following axioms for a {\it projection constant} $\theta \geq 0$, where we set $d_{Y}(X,Z) = \diam(\pi_{Y}(X)\cup\pi_{Y}(Z))$ for any $X,Z \in \cY - \{Y\}$.
  \begin{enumerate}
    \item[(P1)] $\diam(\pi_{Y}(X)) \leq \theta$ for all $X \neq Y$,
    \item[(P2)] (the Behrstock inequality) if $d_Y(X,Z) > \theta$, then $d_X(Y,Z) \leq \theta$, and 
    \item[(P3)] for any $X, Z$ the set $\{Y \in \cY - \{X,Z\}\, |\, d_Y(X,Z) > \theta\}$ is finite.
  \end{enumerate}
  We then say that the collection $(\cY,\{\pi_{Y}\})$ satisfies the {\it projection axioms}. We call the set of functions $\{d_{Y}\}$ the {\it projection distances}.
  
  If Axiom~(P2) is replaced with 
  \begin{enumerate}
    \item[(P2+)] if $d_{X}(Y,Z)>\theta \Rightarrow
          d_{Y}(Z,W)=d_{Y}(X,W)$ for all $X, Y, Z, W$
          distinct,\footnote{One can replace this with an even stronger
          axiom that $d_X(Y,Z)>\theta$ implies $\pi_Y(X)=\pi_Y(Z)$.}
  \end{enumerate}
  then we say that the collection $(\cY,\{\pi_{Y}\})$ satisfies the {\it strong projection axioms}.
\end{defn}

Bestvina--Bromberg--Fujiwara--Sisto~\cite{bestvinabrombergfujiwarasisto} proved that one can upgrade a collection satisfying the projection axioms to a collection satisfying the strong projection axioms as follows.

\begin{thm}[{\cite[Theorem 4.1]{bestvinabrombergfujiwarasisto}}]  \label{thm:projupgrade}
  Assume that $(\cY,\{\pi_{Y}\})$ satisfies the projection axioms with projection constant $\theta$. Then, there are $\{\pi_{Y}'\}$ satisfying the strong projection axioms with projection constant $\theta' = 11\theta$ and such that $d_{Y}-2\theta \leq d_{Y}' \leq d_{Y} + 2\theta$, where $\{d_{Y}\}$ and $\{d_{Y}'\}$ are the projection distances coming from $\{\pi_{Y}\}$ and $\{\pi_{Y}'\}$, respectively. 
\end{thm}

\begin{defn}[Projection complex] \label{def:proj_complex}
  Let $\cY$ be a set that satisfies the strong projection axioms with respect to a constant $\theta \geq 0$. Let $K \in \N$. The {\it projection complex} $\cP = \cP(\cY, \theta, K)$ is a graph with vertex set~$V\cP$ in one-to-one correspondence with elements of $\cY$. Two vertices $X$ and $Z$ are connected by an edge if and only if $d_Y(X,Z) \leq K$ for all $Y \in \cY$.
\end{defn}

Throughout this paper, given a collection satisfying the projections axioms we will always apply Theorem \ref{thm:projupgrade} to upgrade our collection to satisfy the strong projection axioms, unless specified otherwise. We first prove that the projection axioms hold for $\theta$ and then upgrade, but still label the projection constant $\theta$ instead of $\theta'$ by a slight abuse of notation. We will also assume that $K \geq 3\theta$ for the upgraded $\theta$. 

For the rest of this section, we will follow \cite{bestvinabrombergfujiwarasisto}. We refer to Sections 2 and 3 of \cite{bestvinabrombergfujiwarasisto} for any proofs that are omitted in the following.
One virtue of strong projection axioms is that it provides a useful object, called a \textit{standard path}, for studying the geometry of projection complexes. To define it,
for any $X,Z \in \cY$ we consider the set $\cY_{K}(X,Z)$ defined as
\[
  \cY_{K}(X,Z):= \{Y \in \cY - \{X,Z\}| d_{Y}(X,Z)>K\}.
\]
This set $\cY_{K}(X,Z)$ is finite by (P3).
The elements of $\{X\} \cup \cY_{K}(X,Z) \cup \{Z\}$ can be totally ordered
in a natural way, so that each pair of adjacent spaces is connected by an edge in the projection complex $\cP=\cP(\cY,\theta,K)$. The set $\{X\} \cup \cY_K(X,Z) \cup \{Z\}$
is a path between $X$ and $Z$, which we define as the \textit{standard path} between $X$ and $Z$. In particular, this implies that the projection complex $\cP$ is connected.

We can also concatenate two standard paths to make another standard path, as long as the `angle' between the two standard paths is large enough.
\begin{lemma}[Concatenation]
  \label{lem:concatenation}
  If $d_{Y}(X,Z)>K$, then the concatenation of $\cY_{K}(X,Y)$ followed by $\cY_{K}(Y,Z)$ is the standard path $\cY_{K}(X,Z)$.
\end{lemma}
\begin{proof}
  Suppose $d_{Y}(X,Z)>K$. Then by definition $Y\in \cY_{K}(X,Z)$. Let $X'$ be any vertex in the standard path $\cY_{K}(X,Y)$. Then $d_{X'}(X,Y)>K$, so $d_{Y}(X',Z)=d_{Y}(X,Z)>K$, which further implies $d_{X'}(Z,X)=d_{X'}(Y,X)>K$, so $X' \in \cY_{K}(X,Z)$. Similarly, for any vertex $Y'$ in $\cY_{K}(Y,Z)$, we can show that $Y' \in \cY_{K}(X,Z)$, concluding the proof.
\end{proof}

The following lemma says that triangles whose sides are standard paths are nearly tripods.

\begin{lemma}[Standard triangles are nearly-tripods, \cite{bestvinabrombergfujiwarasisto} Lemma 3.6]
  \label{lemma:tripod}
  For every $X,Y,Z \in \cY$, the path $\cY_{K}(X,Z)$ is contained
  in $\cY_{K}(X,Y) \cup \cY_{K}(Y,Z)$ except for at most two vertices. Moreover, in case that there are two such vertices, they are consecutive.
\end{lemma}

This lemma is used to show that standard paths also form quasi-geodesics in the projection complex. 

\begin{lemma}[Standard paths are quasi-geodesics, \cite{bestvinabrombergfujiwarasisto} Corollary 3.7] \label{lem:spqg} 
  Let $X \neq Z$ and let $n = |\cY_{K}(X,Z)| + 1$. Then $\left\lfloor \frac{n}{2} \right\rfloor + 1 \leq d_{\cP}(X,Z) \leq n$. 
\end{lemma}



The next lemma will be used to prove bounded geodesic image theorem (Theorem \ref{thm:bgi}).

\begin{lemma}
  \label{lem:dist4}
  Let $X,Z \in \cY$ be adjacent points in a projection complex. If $Y\in \cY$ satisfies $d_{\cP}(Y,X)\ge 4$ and $d_{\cP}(Y,Z)\ge 4$, then $d_{Y}(X,W)=d_{Y}(Z,W)$ for every $W \in \cY -\{Y\}$.
\end{lemma}

\begin{proof}
  By Lemma \ref{lemma:tripod}, $\cY_{K}(X,Y)$ and $\cY_{K}(Y,Z)$ share at least one vertex, call it $Q$. Then by definition we have $d_{Q}(X,Y)>K$ and $d_{Q}(Z,Y)>K$. Using
  (P2+) twice, for arbitrary $W \in \cY - \{Y\}$, we have
  \[
    d_{Y}(X,W) = d_{Y}(Q,W) = d_{Y}(Z,W),
  \]
  as desired.
\end{proof}

Now we prove the bounded geodesic image theorem for projection complexes, used in Section~\ref{sec:canoe}. We include a proof in the case that the collection $(\cY,\{d_{Y}\})$ satisfies the strong projection axioms, as we will make explicit use of the constant obtained. The result holds with a different constant for the standard projection axioms by \cite[Corollary 3.15]{bestvinabrombergfujiwara}.

\begin{thm}[Bounded Geodesic Image Theorem]  \label{thm:bgi}
  If $\cP = \cP(\cY,\theta,K)$ is a projection complex obtained from a collection $(\cY,\{d_{Y}\})$ satisfying the strong projection axioms and $\gamma$ is a geodesic in $\cP$ that is disjoint from a vertex $Y$, then $d_Y(\gamma(0), \gamma(t)) \leq M$ for all $t$, where $M = 8K+2\theta$.
\end{thm}

\begin{figure}[h]
  \centering
  \begin{overpic}[scale=.7,tics=5]{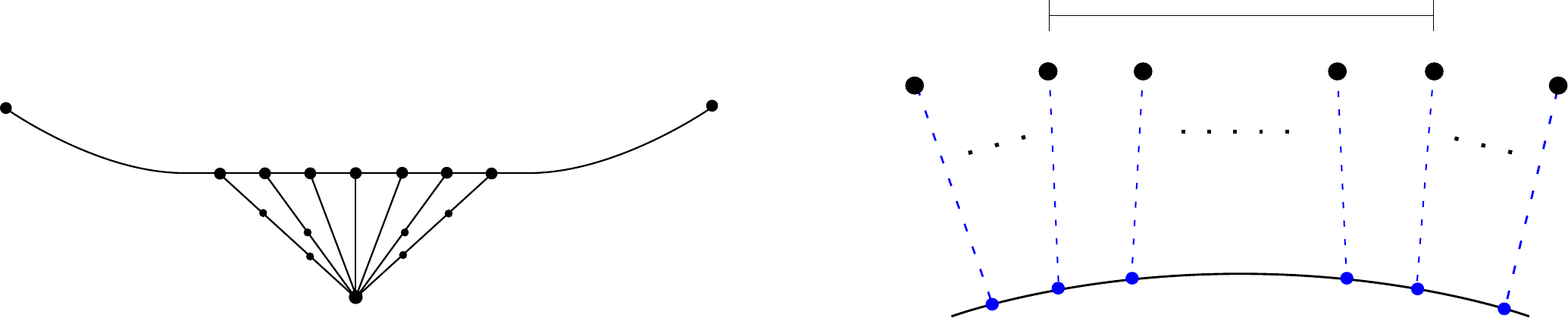}
    \put(3,9){$\gamma$}
    \put(20,0){$Y$}
    \put(13,10.5){\small{$X_i$}}
    \put(30,10.5){\small{$X_j$}}
    \put(58,1){$Y$}
    \put(58,16.5){\small{$X_0$}}
    \put(67,16.5){\small{$X_{i-1}$}}
    \put(73.5,16.5){\small{$X_i$}}
    \put(86,16.5){\small{$X_j$}}
    \put(92,16.5){\small{$X_{j+1}$}}
    \put(98,16.5){\small{$X_n$}}
    \put(77,20){$\le 8$}
    \put(65,-1){\small{$\theta$}}
    \put(69,0){\small{$K$}}
    \put(87,0){\small{$K$}}
    \put(92,-1){\small{$\theta$}}
  \end{overpic}
  \caption{{\small The bound in the Bounded Geodesic Image Theorem is given by considering the configurations above. The geodesic $\gamma$ is shown on the left, and projections onto $Y$ are depicted on the right.}}
  \label{fig:bgit}
\end{figure}
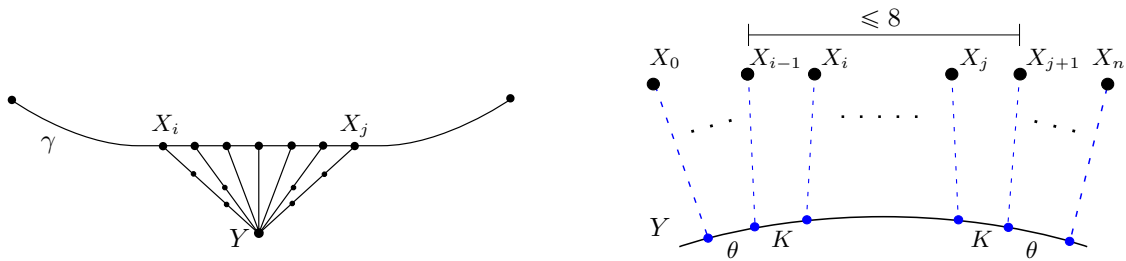

\begin{proof}
  Let $\gamma = \{X_{0},\ldots,X_{n}\}$ be a geodesic in $\cP$ disjoint from a vertex $Y$.
  If $\gamma$ is disjoint from the closed ball of radius 3 about $Y$, then by Lemma~\ref{lem:dist4}:
  \[
    d_{Y}(X_{0},X_{n}) = d_{Y}(X_{1},X_{n}) = \ldots = d_{Y}(X_{n-1},X_{n}) = d_{Y}(X_{n},X_{n}) \underset{(P1)}{\le} \theta.
  \]
  Now assume $\gamma$ intersects the closed ball $B$ of radius 3 about $Y$, and let $X_{i}$ be the first vertex that intersects $B$ and $X_{j}$ be the last one intersects $B$.
  Then $d_{Y}(X_{0},X_{i-1}) \le \theta$ and $d_{Y}(X_{j+1},X_{n}) \le \theta$ as in the first case. Now, by our choice $d_{\cP}(X_{i-1},X_{j+1}) \le d_{\cP}(X_{i-1},Y) + d_{\cP}(Y,X_{j+1}) \le  8$. Also for each $k$ such that $i-1 \le k \le j$, we have $d_{Y}(X_{k},X_{k+1})\le K$ as $X_{k}$ and $X_{k+1}$ are adjacent. Therefore,
  \[
    d_{Y}(X_{0},X_{r}) \le 2\theta + d_{Y}(X_{i-1},X_{j+1}) \le 2\theta + 8K, \text{ for all } r=0, \ldots, n.
  \]
\end{proof}

We will not use the following theorem, but include it here for completeness. An analogous statement for the standard projection axioms was shown in~\cite{bestvinabrombergfujiwara}. The strong projection axiom case along with the specific bound on $K$ recorded here was given by~\cite{bestvinabrombergfujiwarasisto}.

\begin{thm}[{\cite{bestvinabrombergfujiwara,bestvinabrombergfujiwarasisto}}]
  Let $\cY$ be a set that satisfies the strong projection axioms with respect to $\theta \geq 0$. If $K \geq 3\theta$, then the projection complex $\cP(\cY, \theta,K)$ is quasi-isometric to a simplicial tree.  
\end{thm}

\subsection{Spinning subgroups and the result of Clay--Mangahas--Margalit}

\begin{defn}[{\cite[Section 1.7]{claymangahasmargalit}}] \label{def:spinning}
  Let $\cP$ be a projection complex, and let $G$ be a group acting on
  $\cP$. For each vertex $c$ of $\cP$, let $G_c$ be a subgroup of the
  stabilizer of $c$ in $\cP$. Let $L >0$. The family of subgroups
  $\{G_c\}_{c \in V\cP}$ is an {\it (equivariant) $L$-spinning family}
  of subgroups of $G$ if it satisfies the following two
  conditions. 
  \begin{enumerate}
    \item (Equivariance.) If $g \in G$ and $c$ is a vertex of $\cP$, then \[gG_cg^{-1} = G_{gc}. \]
    \item (Spinning condition.) If $a$ and $b$ are distinct vertices of $\cP$ and $g \in G_a$ is non-trivial, then
          \[d_a(b,gb) \geq L. \]
  \end{enumerate}
\end{defn}

\begin{thm}[{\cite[Theorem 1.6]{claymangahasmargalit}}] \label{thm:CMM}
  Let $\cP$ be a projection complex, and let $G$ be a group acting on~$\cP$. There exists a constant $L = L(\cP)$ with the following property. If $\{G_c\}_{c \in V\cP}$ is an $L$-spinning family of subgroups of $G$, then there is a subset $\cO$ of the vertices of $\cP$ so that the normal closure in $G$ of the set $\{G_c\}_{c \in V\cP}$ is isomorphic to the free product $\Asterisk_{c \in \cO} G_c$.
\end{thm}

\begin{remark}
  The constant $L$ is linear in $\theta$. See \cite[Section 6(Proof of Theorem 1.6) and Section 3.1]{claymangahasmargalit}.
\end{remark}

We will also need the following lemma. 

\begin{lemma} \label{lem:L_Lprime}
  Suppose that $\cP = \cP(\cY,\theta,K)$ is a projection complex obtained from a collection $(\cY,\{d_{Y}\})$ satisfying the projection axioms. Let $\cP' = \cP'(\cY,\theta',K')$ be the projection complex obtained from upgrading this collection to a new collection $(\cY,\{d_{Y}'\})$ satisfying the strong projection axioms via Theorem \ref{thm:projupgrade}. If $\{G_{c}\}_{c \in V\cP}$ is an $L$-spinning family of subgroups of $G$ acting on $\cP$, then it is an $L'$-spinning family of subgroups of $G$ acting on $\cP'$ where $L'=L-2\theta$. 
\end{lemma}

\begin{proof}
  By Theorem \ref{thm:projupgrade}, $d_{Y}' \geq d_{Y} -2\theta$ for all $Y\in\cY$. 
\end{proof}

\subsection{Projections in a $\delta$-hyperbolic space} \label{sec:deltahyp}

In this paper we use the $\delta$-thin triangles formulation of
$\delta$-hyperbolicity given as follows. (See \cite[Section
III.H.1]{bridsonhaefliger} and \cite[Section 11.8]{DrutuKapovich} for additional
background.) Given a geodesic triangle $\Delta$ there is an isometry from the set
$\{a,b,c\}$ of vertices of $\Delta$ to the endpoints of a metric tripod
$T_{\Delta}$ with pairs of edge lengths corresponding to the side lengths
of~$\Delta$. This isometry extends to a map
$\chi_{\Delta}: \Delta \rightarrow T_{\Delta}$, which is an isometry when
restricted to each side of $\Delta$. The points in the pre-image of the central
vertex of $T_{\Delta}$ are called the {\it internal points} of~$\Delta$. The
internal points are denoted by $i_{a}$, $i_{b}$, and $i_{c}$, corresponding to the
vertices of $\Delta$ that they are opposite from; that is, the point $i_{a}$ is
on the side $bc$ and likewise for the other two. We say that two points on the triangle are in the same \emph{cusp} if they lie on the segments $[a,i_{b}]$ and $[a,i_{c}]$, or on the analogous segments for the other vertices of the triangle. The triangle $\Delta$ is {\it
  $\delta$-thin} if $p,q \in \chi_{\Delta}^{-1}(t)$ implies that
$d(p,q) \leq \delta$, for all $t \in T_{\Delta}$. In a $\delta$-thin triangle two points lie in the same cusp if they are more than $\delta$ away from the third side. A geodesic metric space is
{\it $\delta$-hyperbolic} if every geodesic triangle is $\delta$-thin. 

Note that another common definition of $\delta$-hyperbolicity requires that every geodesic triangle in the metric space is {\it $\delta$-slim}, meaning that the $\delta$-neighborhood of any two of its sides contains the third side.
A $\delta$-thin triangle is $\delta$-slim; thus, if $X$ is $\delta$-hyperbolic with respect to thin triangles, then $X$ is $\delta$-hyperbolic with respect to slim triangles. We use this fact, as some the constants in the lemmas below are for a $\delta$-hyperbolic space defined with respect to $\delta$-slim triangles.

\begin{defn}
  Let $X$ be a metric space and let $A$ be a closed subset of $X$. For $x \in X$ a {\it nearest-point projection} $\pi_A(x)$ of $x$ to $A$ is a point in $A$ that is nearest to $x$.
\end{defn}

\begin{notation}
  Let $X$ be a metric space and $a,b,p \in X$. We use $[a,b]$ to denote a geodesic from $a$ to~$b$. If $\gamma$ is a path in $X$, we use $\ell(\gamma)$ to denote the length of $\gamma$.
  For $R \geq 0$, we use $B_R(p)$ to denote the open ball of radius $R$ around the point $p$.
\end{notation}

\begin{lemma}[{\cite[Lemma 11.64]{DrutuKapovich}}] \label{lemma_outside_ball}
  Let $X$ be a $\delta$-hyperbolic geodesic metric space. If $[x,y]$ is a geodesic of length $2R$ and $m$ is its midpoint, then every path joining $x$ and $y$ outside the ball $B_R(m)$ has length at least $2^{\frac{R-1}{\delta}}$.
\end{lemma}

\section{A projection complex built from a very rotating family} \label{sec:build_proj}

In this section we construct a projection complex from a fairly rotating family. Throughout, let $G$ be a group that acts by isometries on a $\delta$-hyperbolic metric space~$X$. Let $\cC = (C, \{G_c \, | \, c \in C\})$ be a $\rho$-separated fairly rotating family for some $\rho \geq 20\delta$.
\begin{defn} [Projections] \label{def:rel_dis_fun}
  Let $2+2\delta \leq R \leq \frac{\rho}{2} - 3\delta$. For $p \in
  C$ let $S_{p} = \partial B_{R}(p)$ equipped with the restriction
  of the path metric on $d_{X \setminus B_{R}}(p)$, where two points
  are at infinite distance if they are in different path components
  of $X \setminus B_{R}(p)$. Set $\cY = \{S_{p}\}_{p \in C}$ and,
  for each $a \in C \setminus \{p\}$, let $\pi_{p}(a) \subset S_{p}$
  be the set of nearest point projections of $a$ to $\partial
  B_{R}(p)$ (equivalently, $\pi_p(a)$ consists of intersection
  points of geodesics $[p,a]$ with $\partial
  B_{R}(p)$). 
\end{defn}

We think of the associated projection distances, $d_p(b,c) =
\diam(\pi_{p}(b) \cup \pi_{p}(c))$, as the penalty (up to an error
of a fixed multiple of $\delta$) of traveling from $b$ to $c$
avoiding a ball of fixed radius around $p$.

The aim of this section is to prove the following theorem. 

\begin{thm} \label{thm:proj_com_from_DGO} For $\theta \geq 121 \delta$, the group $G$ acts by isometries on a projection complex associated to the family
  $\left(\cY, \{\pi_{p}\}_{p\in C} \right)$ satisfying the strong projection axioms for $\theta$. Moreover, the
  family of subgroups $\{G_c\}_{c \in C}$ is an $L$-spinning family for
  $L =2^{\frac{R-2}{\delta}} -  4 -248\delta$.
\end{thm}

We prove the projection axioms are satisfied in Subsection~\ref{subsec:axioms}, and we verify the spinning condition in Subsection~\ref{subsec:spinning}.

\subsection{Verification of the projection axioms} \label{subsec:axioms}

\begin{lemma} \label{P1}
  Axiom (P1) holds for any $\theta \geq 4\delta$.
\end{lemma}
\begin{proof}
  Let $p,a \in C$ be distinct and let $a',a''$ be two points in
  $\pi_p(a)$. Then $a'$ and $a''$ lie on two geodesics $\gamma'$
  and $\gamma''$ from $a$ to $p$ such that $a' = \gamma' \cap
  \partial B_{R}(p)$ and $a'' = \gamma''\cap \partial
  B_{R}(p)$. Since geodesics in a $\delta$-hyperbolic space
  $2\delta$-fellow travel (see e.g. \cite[Chapter III.H Lemma
  1.15]{bridsonhaefliger}), we can find a path in $X \setminus
  B_{R}(p)$ of length at most $4\delta$ connecting $a'$ and
  $a''$ by traversing along $\gamma'$ from $a'$ to $a$ a
  distance of $\delta$, then traversing a path of length at most
  $2\delta$ from $\gamma'$ to $\gamma''$, and finally traversing
  along $\gamma''$ a distance of at most $\delta$ back towards
  $a''$. If $d(p,a')<\delta$ then $d(a',a'')<2\delta$. Thus we see that $\diam(\pi_{p}(a)) \leq 4\delta$.
\end{proof}

To prove the remaining axioms we need the following lemma 
\begin{lemma} \label{boundedproj}
  For any $a, b \in X$ and $c \in C$ such that some geodesic $\gamma$ from $a$ to $b$ does not intersect $B_{R + 2\delta}(c)$, we have $d_c(a,b) \leq 4 \delta$. 
\end{lemma}

\begin{proof}
  Let $a'$ and $b'$ be points in $\pi_c(a)$ and $\pi_c(b)$,
  respectively. Consider the triangle formed by $\gamma$ and
  geodesics $[a,c]$ and $[b,c]$ where $a' \in [a,c]$ and $b' \in
  [b,c]$. Let $a''$ be the point on $[a,c]$ outside of $B_R(c)$ at
  distance $\delta$ from $a'$, and define $b''$ analogously. See
  Figure \ref{fig:Lemma3_4}. By hypothesis, $a''$ and $b''$ are more
  than $\delta$ away from $\gamma$ so $a''$ and $b''$ must be in the
  same cusp of the geodesic triangle. Therefore, $d(a'', b'') \leq
  \delta$. Note that any geodesic $[a'',b'']$ misses $B_R(c)$, so it
  follows that $d_{X \setminus B_R(c)} (a',b') \leq 3 \delta$ by
  concatenating geodesics $[a', a''], [a'',b'']$, and $[b'',
  b']$. Now by Lemma \ref{P1}, we see that $\diam(\pi_c(a) \cup
  \pi_c(b)) \leq 4 \delta$.
\end{proof}

\begin{figure}[h]
  \centering
  \def\svgwidth{\columnwidth}
  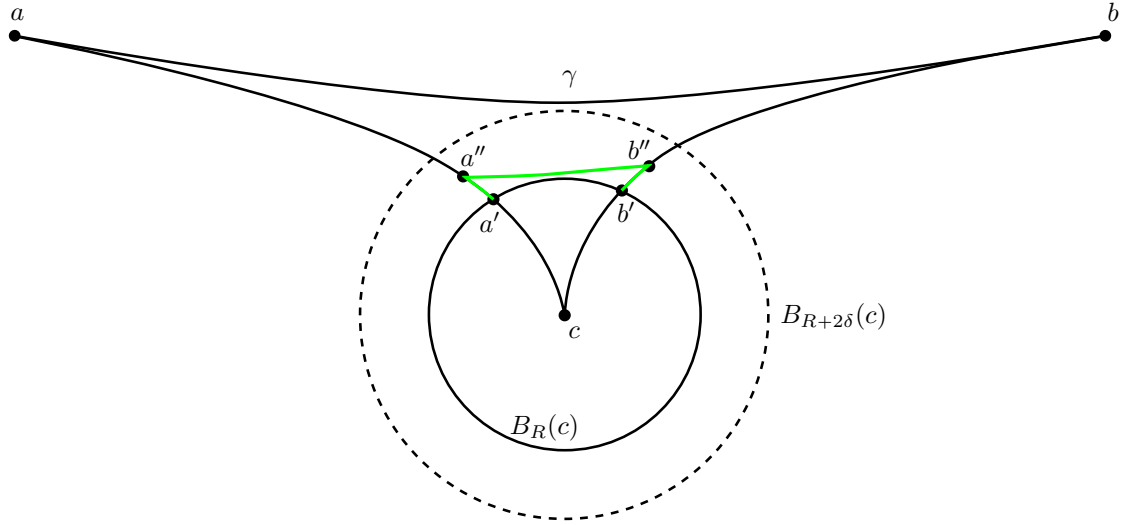
  \caption{Configuration of points in Lemma \ref{boundedproj}. The path in green is a path of length at most $3\delta$ between $a'$ and $b'$ which misses $B_{R}(c)$.}
  \label{fig:Lemma3_4}
\end{figure}

\begin{lemma}
  Axiom (P2) holds with respect to $\{d_a \, | \, a \in C\}$ and $\theta \geq 4 \delta$.
\end{lemma}
\begin{proof}
  Suppose $d_a(b,c) > \theta$; we will show $d_b(a,c)\leq\theta$. By Lemma
  \ref{boundedproj}, every geodesic $[b, c]$ intersects $B_{R+2\delta}(a)$.
  Using the same lemma, we are done if we show some geodesic $[a,c]$ avoids
  $B_{R+2\delta}(b)$. Let $a^{\prime}$ be a nearest point projection of $a$
  to $[b,c]$, and let $[a^{\prime}, c] \subset [b,c]$ be the subpath from
  $a^{\prime}$ to $c$. Note that $a'$, and therefore any geodesic $[a, a']$,
  is contained in $B_{R+2\delta}(a)$. Suppose $[a,c]$ and $[a,a^{\prime}]$ are
  any geodesics and consider the geodesic triangle formed by them and
  $[a^{\prime}, c]$. Using the fact that the points in $C$ are at least $\rho$-separated,
  we see that for any $x \in [a,a']\cup[a^{\prime}, c]$ we have
  $d(b, x) > \rho -(R+2\delta) \ge R + 4\delta$. If $x \in [a, a']$ then $d(b, x) \geq d(b, a) - d(a, x) > \rho -(R+2\delta)$. If $x \in [a', c]$, then $d(b, x) = d(b, a') + d(a', x) \geq d(b, a')$, and we just showed this quantity was greater than $\rho -(R+2\delta)$. The segment $[a,c]$ must be
  contained in the union of $\delta$-neighborhoods of the other two sides, and
  thus, no point on $[a,c]$ can be $(R+2\delta)$-close to $b$.
\end{proof}

\begin{lemma}
  Axiom (P3) holds with respect to $\{d_a \, | \, a\in C\}$ and $\theta \geq 4 \delta$.
\end{lemma}
\begin{proof}
  Let $b,c \in C$. We must show the set $\{a \, | \, d_a(b,c)>\theta\}$ is
  finite. If $d_a(b,c)>\theta$, then by Lemma \ref{boundedproj} each geodesic $[b,c]$ 
  must intersect $B_{R+2\delta}(a)$. Fix a geodesic $[b,c]$, and cover $[b,c]$ with finitely many segments of length $\frac{1}{2}$. Each element
  of $\{a \, | \, d_a(b,c)>\theta\}$ lies in a $(R+2\delta)$-neighborhood of one of
  these segments. Since $\rho \ge 2R+6\delta> 2(R + 2\delta)$, each $(R+2\delta)$-neighborhood of such a segment contains at most one point in the set $\{a \, | \, d_a(b,c)>\theta\}$. Thus, the set $\{a \, | \, d_a(b,c)>\theta\}$ is finite.
\end{proof}

\subsection{Verification of the spinning family conditions} \label{subsec:spinning}

For the remainder of this section, let~$\cP$ be the projection complex associated to the set $C$ and the projection distance functions $\{d_{p}|p \in C\}$. The group $G$ acts by isometries on $\cP$. By the construction of $\cP$, for all $c \in C$, the group~$G_c$ is a subgroup of the stabilizer of the vertex $c$ in $\cP$. Moreover, the equivariance condition, Definition~\ref{def:spinning}(1), follows from Definition~\ref{def:very_rot}(a-3). The next lemma verifies the spinning condition, Definition~\ref{def:spinning}(2).


\begin{lemma} \label{lem:spinning}
  If $a, b \in V\cP$ and $g \in G_a$ is non-trivial, then $d_a(b,gb) \geq 2^{\frac{R-2}{\delta}} -  4 -6\delta$. 
\end{lemma}
\begin{proof}
  Let $a,b \in V\cP$, and let $g \in G_a$ be non-trivial. Let
  $\sigma$ be a geodesic in $X$ from $b$ to $gb$. Let $p_1$ and
  $p_2$ be closest point projections of $b$ and $gb$
  respectively to $\partial B_R(a)$. By the fairly rotating
  condition, $\sigma$ passes through a point $a'$ in the
  $1$-neighborhood of $a$. Let $q_1$ and $q_2$ be the
  intersection points of $\sigma$ with $\partial B_{R-1}(a')$,
  and let $\gamma$ be a path from $p_1$ to $p_2$ in $X \setminus
  B_{R}(a)$. We will now construct, using $\gamma$, a path $\gamma'$ from $q_1$ to $q_2$ in
  $X \setminus B_{R-1}(a')$. See Figure \ref{fig:spinning}. A
  lower bound on the length of $\gamma'$ from Lemma
  \ref{lemma_outside_ball} will give us a lower bound on the
  length of $\gamma$.

  Consider the triangle in $X$ formed by $\sigma$ and geodesics
  $[b,a]$ and $[gb,a]$ such that $p_1 \in [b,a]$ and $p_2 \in
  [gb, a]$. Let $q'_1$ be the point on $\sigma \cap B_R(a)$ we
  reach by following $\sigma$ away from $q_1$ towards
  $b$. Define $q'_2$ similarly. The points $q'_1$ and $p_1$ are
  in the same cusp of the geodesic triangle with vertices $b,a,$
  and $a'$. This follows since $d_{X}(p_{1},a) = R$,
  $d_{X}(q_{1}',a') \geq R-1$, $[a,a']$ is an edge of the
  triangle of length $1$, and $R \geq 2+2\delta$. Note also that
  $d_X(b,p_1)=d_X(b,a)-R$ and $d_X(b,a)-1\leq d_X(b,a')\leq
  d_X(b,a)+1$, so $d_X(b,a)-R\leq d_X(b,q_1)\leq d_X(b,a)-R+2$.
  Thus, we can travel a distance $\leq 2$ from $q_1$ towards $b$
  to get to a point at the same distance from $b$ as $p_1$, and
  then along each side of the triangle and $\delta$ between the
  sides to see $d_{X \setminus B_{R-1}(a')}(p_1, q_1) \leq 2+3
  \delta$, and similarly for $p_2$ and $q_2$. By concatenating
  $\gamma$ with paths outside $B_{R-1}(a')$ of length at most $2
  + 3\delta$ from $p_1$ to $q_1$ and $p_2$ to $q_2$ we see that
  $\ell(\gamma') \leq \ell(\gamma) + 4 + 6 \delta$.
  
  Now by Lemma~\ref{lemma_outside_ball},  we have $\ell(\gamma') \geq 2^{\frac{R-2}{\delta}}$; in the language of the lemma, $[q_1,q_2] \subset \sigma$ is a geodesic of length $2(R-1)$, $\gamma'$ is a path connecting $q_1$ and $q_2$ outside the ball $B_{R-1}(a')$, and $a'$ is the midpoint of the geodesic segment. Therefore, $\ell(\gamma) \geq 2^{\frac{R-2}{\delta}} -  4 -6\delta$.
\end{proof}

\begin{figure}[h]
  \centering
  \def\svgwidth{\columnwidth}
  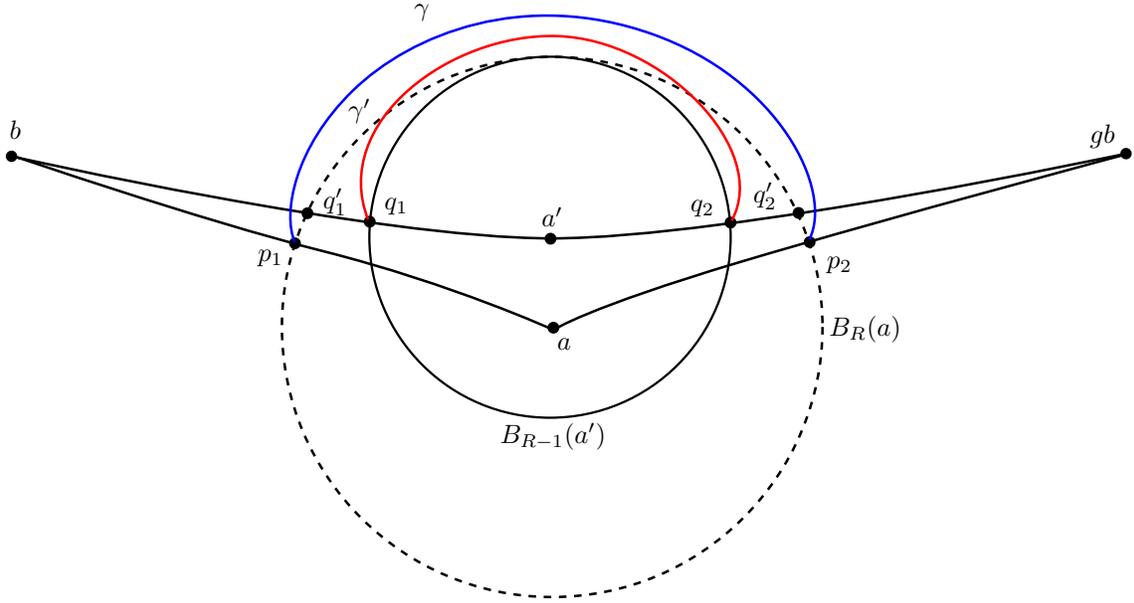
  \caption{Setup for Lemma \ref{lem:spinning}. The path in blue is $\gamma$, a geodesic in $X \setminus B_{R}(a)$ from $p_{1}$ to $p_{2}$ and the path in red is $\gamma'$, a geodesic in $X \setminus B_{R-1}(a')$ from $q_{1}$ to $q_{2}$.}
  \label{fig:spinning}
\end{figure}

We conclude this section with:
\begin{proof}[Proof of Theorem~\ref{thm:proj_com_from_DGO}]
  The lemmas in Subsection~\ref{subsec:axioms} combine to prove the projection axioms hold with respect to $C$ equipped with the distance functions $\{d_{p}|p \in C\}$. The discussion and lemma in Subsection~\ref{subsec:spinning} along with upgrading the projection axioms to the strong projections axioms via Theorem \ref{thm:projupgrade} and applying Lemma \ref{lem:L_Lprime} prove the remaining claims in the statement of the theorem. 
\end{proof}

\section{Free products from spinning families} \label{sec:canoe}

The aim of this section is to give a new proof of Theorem~\ref{thm:CMM}, the result of Clay--Mangahas--Margalit.

\subsection{Canoeing paths}

The results in this section are motivated by the notion of canoeing in the hyperbolic plane, as illustrated in Figure~\ref{fig_canoeH2}. We will not use the following proposition, but include it as motivation.

\begin{prop}[{\cite[Lemma 11.3.4]{ECHLPT92}, Canoeing in $\Hy^2$}]
  Let $0<\alpha \le \pi$. There exists $L >0$ so that if $\sigma = \sigma_1 * \dots * \sigma_k$ is a concatenation of geodesic segments in $\Hy^2$ of length at least $L$ and so that the angle between adjacent segments is at least $\alpha$, then the path $\sigma$ is a $(K,C)$-quasi-geodesic, with constants depending only on $\alpha$.
\end{prop}

\begin{figure}[h] 
  \centering
  \begin{overpic}[scale=.8, tics=5]{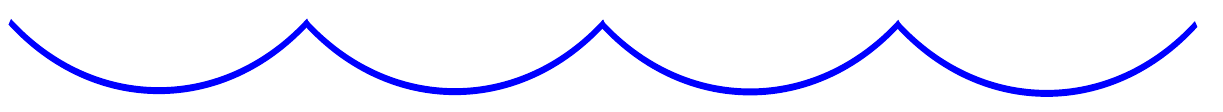}    
    \put(-8,4){$\Hy^2$}
    \put(9,4){\small{$\geq L$}}
    \put(35,4){\small{$\geq L$}}
    \put(58,4){\small{$\geq L$}}
    \put(82,4){\small{$\geq L$}} 
    \put(22,0.5){\small{$\geq \alpha$}}
    \put(47,0.5){\small{$\geq \alpha$}}    
    \put(71,0.5){\small{$\geq \alpha$}}    
  \end{overpic}
  \caption{{\small Canoeing paths in the hyperbolic plane are embedded quasi-geodesics. The segments have length at least $L$, and the angle between adjacent segments is at least $\alpha$.}}
  \label{fig_canoeH2}
\end{figure}
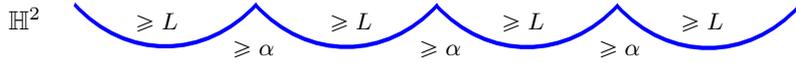

\begin{defn}
  If $\gamma = \{X_1, \ldots, X_k\}$ is a path of vertices in a projection complex, then the {\it angle} in $\gamma$ of the vertex $X_i$ is $d_{X_i}(X_{i-1}, X_{i+1})$.  
\end{defn}

The following definition is tailored to our purposes.

\begin{defn}
  A {\it $C$-canoeing path} in a projection complex is a concatenation $\gamma = \gamma_1 * \gamma_2 * \ldots * \gamma_m$ of paths so that the following conditions hold. 
  \begin{enumerate}
    \item Each $\gamma_i$ is an embedded nondegenerate path, and is either a geodesic or the concatenation $\alpha_i*\beta_i$ of two geodesics.
    \item The common endpoint $V_i$ of $\gamma_i$ and $\gamma_{i+1}$ has angle at least $C$ in $\gamma$ for $i \in \{1, \ldots, m-1\}$. We refer to these points as \emph{large angle points} of $\gamma$. 
  \end{enumerate}
\end{defn}

Since any subpath of a canoeing path is canoeing, it follows
that canoeing paths are embedded.
The proof that the endpoints of a canoeing path are distinct uses the Bounded Geodesic Image Theorem for projection complexes (Theorem~\ref{thm:bgi}).

\begin{prop} \label{prop:canoe_distinct}
  Let $\cP(\cY,\theta,K)$ be a projection complex satisfying the strong projection axioms, and let $M$ be the constant given in Theorem~\ref{thm:bgi}. If $C > 4M+K$, then the the large angle points of a $C$-canoeing path lie on a standard path. In particular, the endpoints of a $C$-canoeing path are distinct. 
\end{prop}
\begin{proof}
  Let $\gamma = \gamma_1 * \ldots * \gamma_k$ be a $C$-canoeing path with $C >4M+K$. Let $x$ and $y$ denote the endpoints of $\gamma$. Let $B_i$ be the vertex of $\gamma_i$ adjacent to the large-angle point $V_i$, and let $B_i'$ be the vertex of $\gamma_{i+1}$ adjacent to $V_i$. We will assume $\gamma_i$ is the concatenation $\alpha_i*\beta_i$ of two geodesics.

  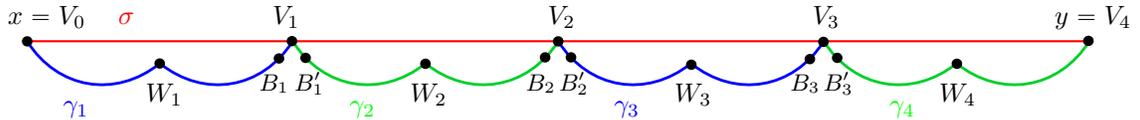
\begin{figure}[h] 
    \centering
    \begin{overpic}[scale=.7,tics=5]{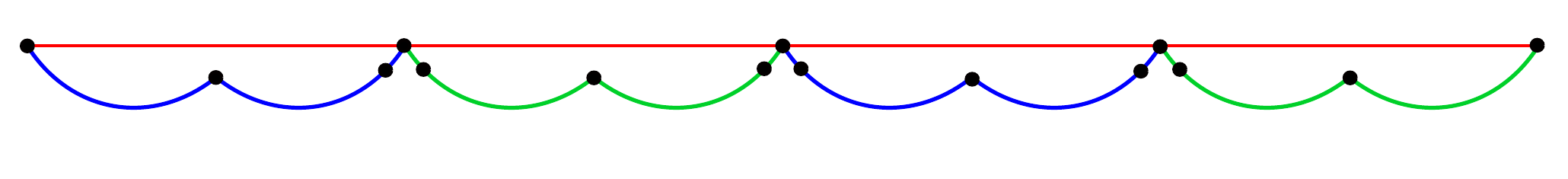}    
      \put(0,10){$x=V_{0}$}
      \put(10,10){\textcolor{red}{$\sigma$}}
      \put(12.5,3){$W_1$}
      \put(24,10){$V_1$}
      \put(22.8,4){\small{$B_1$}}
      \put(26,4){\small{$B_1'$}} 
      \put(36.5,3){$W_2$}
      \put(60.5,3){$W_3$}
      \put(84.5,3){$W_4$}
      \put(95,10){$y=V_{4}$}
      \put(49,10){$V_2$}
      \put(73,10){$V_3$}
      \put(47,4){\small{$B_2$}}
      \put(50,4){\small{$B_2'$}}
      \put(71,4){\small{$B_3$}}
      \put(74,4){\small{$B_3'$}}
      \put(5,2){\textcolor{blue}{$\gamma_1$}}
      \put(31,2){\textcolor{green}{$\gamma_2$}}
      \put(55,2){\textcolor{blue}{$\gamma_3$}}
      \put(80,2){\textcolor{green}{$\gamma_4$}}
    \end{overpic}
    \caption{{\small To prove that the endpoints, $x$ and $y$, of a canoeing path $\gamma$ are distinct, we show that the red path $\sigma$ that connects the large-angle points is a standard path.}}
    \label{fig_canoeing_path}
  \end{figure}

  Write for brevity $V_{0}:=x$ and $V_{k}:=y$.
  For $i \in \{1, \ldots, k\}$, let $\sigma_{i}$ be the standard path from $V_{i-1}$ to $V_{i}$.
  Then let $\sigma = \sigma_1 * \ldots * \sigma_k$ be the
  concatenation of the standard paths. We will show that $\sigma$
  is a nontrivial standard path by proving each concatenation
  angle is larger than $K$, which is a sufficient condition by
  Lemma~\ref{lem:concatenation}. Note that by the Bounded Geodesic Image Theorem (Theorem \ref{thm:bgi}), $d_{V_{i}}(B_{i},V_{i-1}) \le 2M$ and $d_{V_{i}}(B_{i}',V_{i+1}) \le 2M$. By the assumption that $d_{V_{i}}(B_{i},B_{i}')>4M+K$, we have $d_{V_{i}}(V_{i-1},V_{i+1})>K$, concluding the proof.
\end{proof}

Combining this with Lemma \ref{lem:spqg} yields the following.

\begin{corollary} \label{cor:cpbound} 
  Let $\gamma$ be a $C$-canoeing path with $C> 4M+K$ connecting the points $X$ and $Y$ and let $k$ be the number of large angle points on $\gamma$. Then $d_{\cP}(X,Y) \geq \frac{k}{2}$.
\end{corollary}

\subsection{Canoeing in windmills to prove dual graphs are trees} \label{subsec:cmmproof}

We will prove the following theorem in this section.

\begin{thm} \label{thm:new_CMM}
  Suppose that $\cP = \cP(\cY,\theta,K)$ is a projection complex
  satisfying the strong projection axioms, and let $G$ be a
  group acting on $\cP$ preserving the projection data. Suppose
  that $\{G_c\}_{c \in V\cP}$ is an $L$-spinning
  family of subgroups of $G$ for $L > 4M+K$, where $M$ is the
  constant given in Theorem~\ref{thm:bgi}. Then, there is a
  subset $\cO\subset V\cP$ of the vertices of $\cP$ so that the subgroup
  of $G$ generated by $\{G_c\}_{c \in V\cP}$ is isomorphic
  to the free product $*_{c \in \cO} G_c$.
\end{thm}

As in \cite{claymangahasmargalit}, we inductively define a sequence of subgraphs $\{W_i\}_{i \in \N}$ of $\cP$ called {\it windmills}. Our methods diverge from those of Clay--Mangahas--Margalit in that we show that each windmill $W_i$ admits a graph of spaces decomposition with dual graph a tree. We inductively define a sequence of subgroups $\{G_i\}_{i \in \N}$ of $G$ so that $G_i$ acts on the dual tree to $W_i$ with trivial edge stabilizers. Hence, we obtain a free product decomposition for $G_i$ by Bass-Serre theory. By the equivariance condition and because the windmills exhaust the projection complex, we ultimately obtain
\[ \la G_c \ra_{c \in V\cP} = \varinjlim_{i} G_i  = *_{c \in \cO} G_c. \] 

\begin{defn}[Windmills]
  Fix a base vertex $v_0 \in V\cP$, let $\cO_{-1}=\{v_0\}$, and let
  $W_0 = \{v_0\}$ be the base windmill. Let $G_0 = G_{v_0}$. Let
  $N_0$ be the $1$-neighborhood of $W_0$, and let $G_1 = \la \,
  G_v \, | \, v \in N_0 \, \ra$. Recursively, for $k \geq 1$,
  let $W_k = G_k \cdot N_{k-1}$, let $N_k$ be the
  $1$-neighborhood of $W_k$, and let $G_{k+1} = \la \, G_v \, |
  \, v \in N_k \, \ra$. Finally, for $k \ge 0$, let $\cO_k$ be a set of
  $G_k$-orbit representatives in $N_{k}\smallsetminus W_{k}$ and
  $\cO=\cup_{k=-1}^\infty \cO_k$.
\end{defn}

We will use the following notion to extend geodesics in the projection complex. 

\begin{defn} \label{def_perp_bound}
  The {\it boundary} of the windmill $W_k$, denoted by $\p W_k$, is the set of vertices in $W_k$ that are adjacent to a vertex in $\cP - W_k$. A geodesic $[u,v]$ in $\cP$ that is contained in $W_k$ is {\it perpendicular to the boundary at $u$} if $u \in \p W_k$ and $d_{\cP}(v,\p W_k) = d_{\cP}(v,u)$.
\end{defn}

The next lemma follows immediately from Definition~\ref{def_perp_bound}. 

\begin{lemma} \label{lemma_extend}
  If a geodesic $[u,v]$ contained in $W_k$ is perpendicular to the boundary at $u$, and $w \in \cP - W_k$ is a vertex adjacent to $u$, then the concatenation $[v,u]*[u,w]$ is a geodesic in $\cP$.
\end{lemma}

\begin{proof}[Proof of Theorem~\ref{thm:new_CMM}]
  First, we show that the following properties hold for all $k\geq 0$:
  \begin{enumerate}
    \item[(I1)] Any two distinct vertices of $W_k$ can be joined by an
          $L$-canoeing path $\gamma = \gamma_1 * \gamma_2 * \ldots *
          \gamma_m$ in $W_k$ so that the following holds.
          If the initial vertex of $\gamma_1$ is on the boundary of $W_k$, then the first geodesic $\alpha_1$ (or $\gamma_1$) is perpendicular to the boundary at that point. Likewise for the other endpoint of $\gamma$.

    \item[(I2)] Two translates of $N_{k-1}$ either coincide,
          intersect in a point, or are disjoint. The stabilizer in
          $G_k$ of $N_{k-1}$ is $G_{k-1}$ and the stabilizer of $v\in
          N_{k-1}\smallsetminus W_{k-1}$ in $G_k$ is $G_v$. The
          skeleton $S_k$ of the cover of $W_k$ by the translates of
          $N_{k-1}$ is a tree. (See Figure \ref{fig:treeblobs}.) Furthermore, if $\gamma$ is a canoeing
          path constructed in (I1) connecting two vertices of $W_{k}$,
          then every vertex of $\gamma$ which is an intersection point between distinct translates of $N_{k-1}$ is a large angle point of $\gamma$.

  \end{enumerate}
  \begin{figure}[ht!]
    \centering
    \includegraphics[width=0.6\textwidth]{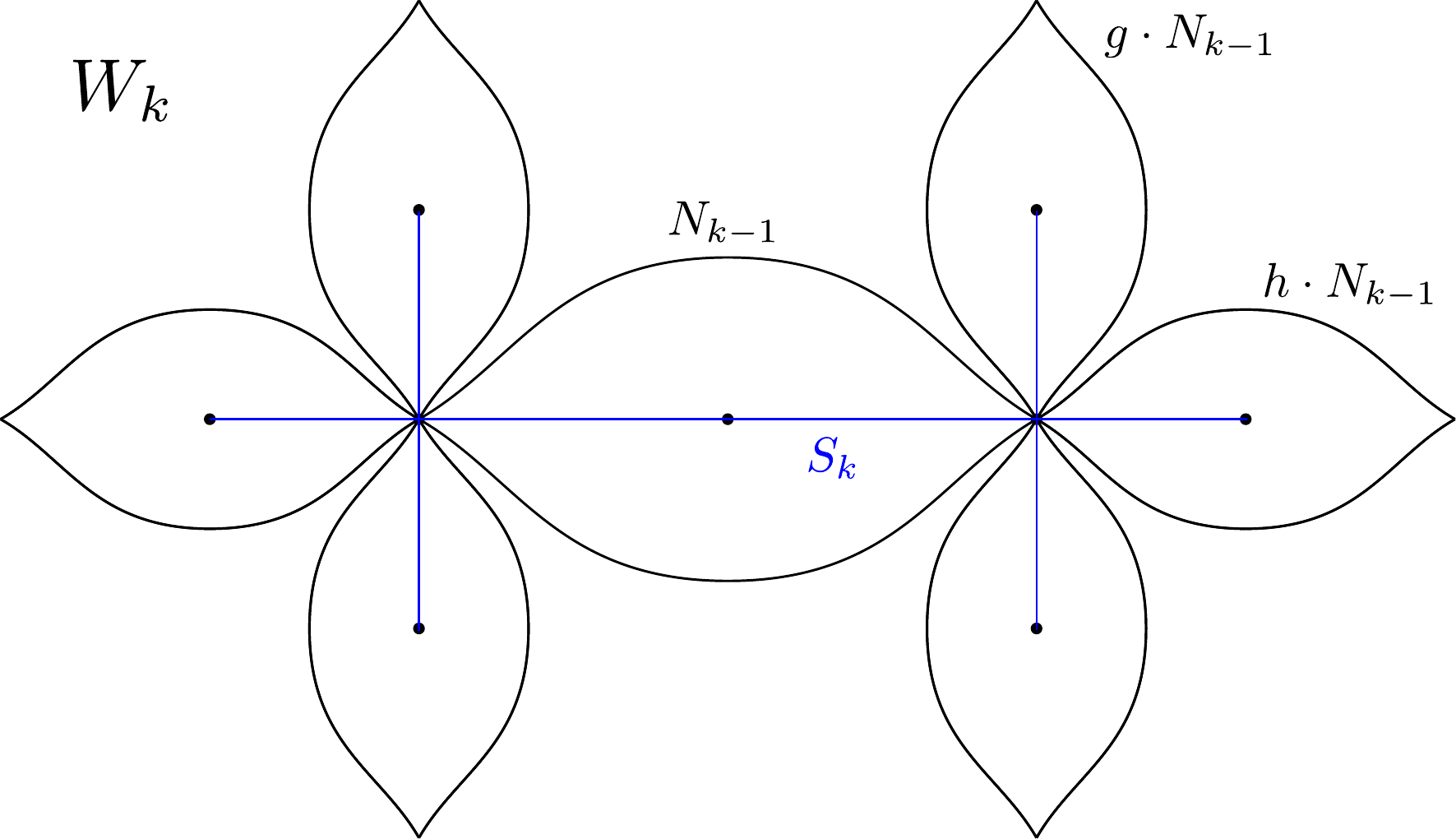}
    \caption{The cover of the windmill $W_k$ by the translates of $N_{k-1}$ and its skeleton $S_k$.}
    \label{fig:treeblobs}
  \end{figure}

  Recall that the {\it skeleton} is defined to be the bipartite graph
  whose vertex set is $V_1\sqcup V_2$ with a vertex $p\in V_1$ for every
  translate of $N_{k-1}$ and a vertex $q\in V_2$ for every intersection
  point between distinct translates, and edges represent incidence.

  We proceed by induction. For the base case, we note that the claims
  hold trivially for $k=0$.
  For the induction hypotheses, assume that (I1) and (I2) hold for
  $k-1\geq 0$; we will prove they also hold for $k$. We will need the following claim. 
  
  \begin{claim} \label{claim:vertex_intersect}
    If $g \in G_v\smallsetminus\{1\}$ for a vertex $v \in N_{k-1} - W_{k-1}$, then $g\cdot N_{k-1} \cap N_{k-1} = \{v\}$. 
  \end{claim}
  \begin{proof}[Proof of Claim~\ref{claim:vertex_intersect}]
    Let $x \in N_{k-1}$ and $y \in g\cdot N_{k-1}$ with $x \neq v \neq y$. To show $x \neq y$, we will build a path from $x$ to $y$ satisfying (I1). See Figure~\ref{figure_induction}. Let $v' \in W_{k-1}$ be adjacent to $v$. Let $x' \in W_{k-1}$ so that $x = x'$ if $x \in W_{k-1}$, and otherwise, $x$ and $x'$ are adjacent. By the induction hypotheses, there exists a path $\gamma = \gamma_1 * \ldots * \gamma_m$ from $x'$ to $v'$ in $W_{k-1}$ satisfying conditions (I1). The first geodesic $\alpha_1$ (or $\gamma_1$) of $\gamma$ extends to a geodesic to $x$ by Lemma~\ref{lemma_extend}. Similarly, the final geodesic $\beta_m$ (or $\gamma_m$) extends to a geodesic to~$v$. Thus, the path $\gamma$ extends to a path $\gamma'$ from $x$ to $v$ that is contained in $N_{k-1}$ and satisfies the conditions of (I1). Similarly, there exists a path $\delta = \delta_1* \ldots *\delta_n$ from $gv'$ to a vertex $y' \in g\cdot W_{k-1}$ with $y'=y$ if $y \in g\cdot W_{k-1}$ or $d_{\cP}(y,y') = 1$. As above, the path $\delta$ extends to a path from $v$ to $y$ satisfying (I1). Since $d_v(v',gv') \geq L$, the concatenation $\gamma_1 * \ldots * \gamma_m * \delta_1 * \ldots * \delta_n$ satisfies (I1). Thus, $x \neq y$ by Proposition~\ref{prop:canoe_distinct}. We also point out that $v$ is a large angle point of this canoeing path.    
    \end{proof}
    
    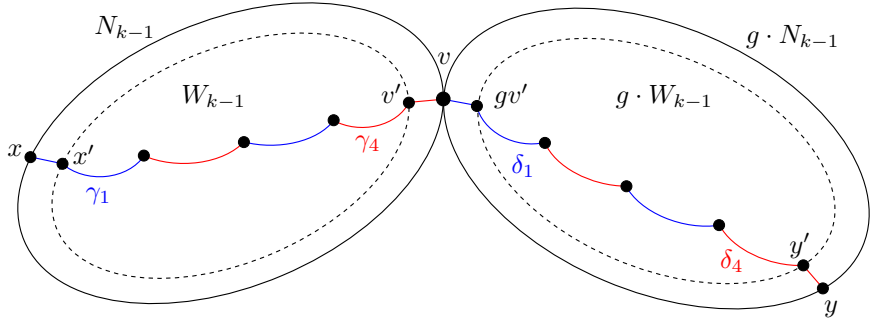
\begin{figure}[h] 
    \centering
     \begin{overpic}[scale=.55, tics=5]{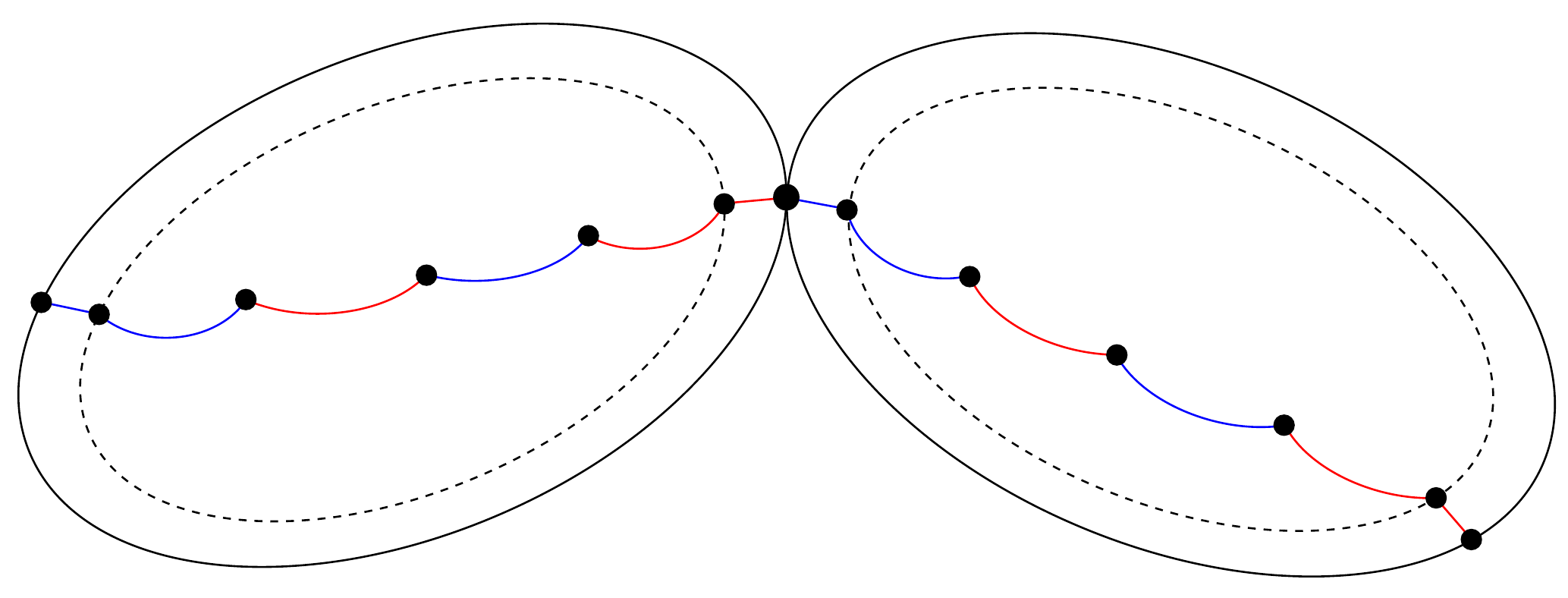}    
        \put(0,19){$x$}
        \put(7.5,18){$x'$}
        \put(10,33){$N_{k-1}$}
        \put(20,25){$W_{k-1}$}
        \put(9,14){\textcolor{blue}{$\gamma_1$}}
        \put(40,20){\textcolor{red}{$\gamma_4$}}
        \put(43,25){$v'$}
        \put(49.5,29.5){$v$}
        \put(56,25){$gv'$}
        \put(58,17){\textcolor{blue}{$\delta_1$}}
        \put(85,32){$g\cdot N_{k-1}$}
        \put(70,25){$g\cdot W_{k-1}$}
        \put(82,6){\textcolor{red}{$\delta_4$}}
        \put(90,7.8){$y'$}
        \put(94,1){$y$}
      \end{overpic}
     \caption{{\small Canoeing paths are used to prove $N_{k-1} \cap g \cdot N_{k-1} = \{v\}$. Canoeing paths $\gamma_1*\ldots *\gamma_m$ from $x'$ to $v'$ and $\delta_1*\ldots *\delta_n$ from $gv'$ to $y'$ exist by the induction hypotheses. Since the ends of these paths are perpendicular to the boundary, they can be extended to a canoeing path from $x$ to $y$. Thus, $x \neq y$ for any $x \in N_{k-1} - \{v\}$ and $y \in g\cdot N_{k-1} - \{v\}$.}}
     \label{figure_induction}
    \end{figure}    
   
   \begin{claim} \label{claim_IH1}
    Given the induction hypotheses, property (I1) holds for $W_k$. 
   \end{claim}
   \begin{proof}[Proof of Claim~\ref{claim_IH1}]
        Let $x,y \in W_k$. Suppose first that $x$ and $y$ are
        contained in the same $G_k$-translate of $N_{k-1}$, say in
        $N_{k-1}$ itself. Let $x', y' \in W_{k-1}$ with $x=x'$ if $x
        \in W_{k-1}$ and $d_{\cP }(x,x')=1$ otherwise, and similarly
        for $y'$. By the induction hypothesis, there exists a path
        $\gamma = \gamma_1 * \ldots * \gamma_m$ from $x'$ to $y'$. The
        first geodesic $\alpha_1$ (or $\gamma_1)$ can be extended to
        $x$ by Lemma~\ref{lemma_extend}, and the last geodesic
        $\beta_m$ (or $\gamma_m$) can be extended to $y$ to produce a
        new geodesic $\gamma'$ that is perpendicular to the boundary
        at $x$ and $y$. Thus, (I1) holds in this case.
        
        We may now assume that $x \in N_{k-1}$ and $y \in g \cdot
        N_{k-1}$ for some $g \in G_k\smallsetminus G_{k-1}$. Choose a
        decomposition $g = g_1 \ldots g_m$ with $g_i \in G_{v_i}$ for
        $v_i \in N_{k-1}$ so that $m$ is minimal. Observe that $m\geq
        1$ and that $g_i \notin G_{k-1}$ for any $i \in \{1, \ldots,
        m\}$. Indeed, if $g_0g_i$ appears as a subword of $g$ with
        $g_0 \in G_{k-1}$ and $g_i \in G_{v_i}$, then $g_0g_i =
        g_0g_ig_0^{-1}g_0 = g_{i'}g_0$ for $g_{i'} \in G_{g_0v_i}$ by
        the equivariance condition. That is, the element $g_0$ can be
        shifted to the right, and since $g_0$ stabilizes $N_{k-1}$,
        the element $g$ could be written with fewer letters,
        contradicting the minimality of the decomposition.
   
        We now build a path from $x$ to $y$. The translates
        $g_1g_2\ldots g_i\cdot N_{k-1}$ and $g_1g_2 \ldots
        g_{i+1}\cdot N_{k-1}$ intersect in the single vertex $g_1g_2
        \ldots g_iv_{i+1}$ for $i \in \{1, \ldots, k-1\}$ by the
        assumptions on $g_i$ and
        Claim~\ref{claim:vertex_intersect}. Similarly, $N_{k-1} \cap
        g_1N_{k-1} = \{v_1\}$. Therefore, the methods in the proof of
        Claim~\ref{claim:vertex_intersect} can be inductively applied
        to build a path from $x$ to $y$ satisfying (I1). That is, the
        path is constructed to pass through each intersection point
        $c_{i+1}=g_1g_2 \ldots g_iv_{i+1}$ and the edges
        $e_{i+1},f_{i+1}$ immediately before and after $c_{i+1}$
        satisfy $f_{i+1}=h_{i+1}(e_{i+1})$ for a nontrivial
        $h_{i+1}\in G_{c_{i+1}}$. The restriction of
        the path to each translate of $N_{k-1}$ is built using
        property (I1) applied to the translate of $W_{k-1}$.
   \end{proof}
          
   \begin{claim} \label{claim_IH2}
    Property (I2) is satisfied by $G_k$ and $W_k$.
   \end{claim}
   \begin{proof}[Proof of Claim~\ref{claim_IH2}]
We may assume one of the translates is $N_{k-1}$ itself and the other
is $g\cdot N_{k-1}$ where $g\in G$ is written as $g=g_1\cdots g_m$
with $g_i\in G_{v_i}$ and $m$ minimal as above. If $m>1$ then the
canoeing path we constructed from a vertex in $N_{k-1}$ to a vertex in
$g(N_{k-1})$ is nondegenerate, showing that $N_{k-1}\cap
g\cdot N_{k-1}=\emptyset$. If $m=1$, we showed in Claim
\ref{claim:vertex_intersect} that $N_{k-1}\cap
g_1\cdot N_{k-1}=\{v_1\}$. We now prove that $S_k$ is a tree. Since $W_k$
is a connected graph, $S_k$ is also connected.

        Suppose towards a contradiction that $p_1, q_1, p_2, q_2,
        \ldots, p_n, q_n, p_1$ is an edge path that is an embedded
        loop in the graph with $p_i \in V_1$ and $q_i \in V_2$. Each
        vertex $p_i$ corresponds to a translate $g_i\cdot N_{k-1}$
        with $g_i \in G_k$. Consecutive translates intersect in a
        point, and since the edge path does not backtrack, the
        intersection points $g_{i-1}\cdot N_{k-1}\cap g_{i}\cdot
        N_{k-1}$ and $g_{i}\cdot N_{k-1}\cap g_{i+1}\cdot N_{k-1}$ are
        distinct. Under these assumptions we constructed a
        nondegenerate canoeing path from any vertex in $g_1\cdot
        N_{k-1}$ to any vertex in $g_n\cdot N_{k-1}$, showing that the
        two translates are disjoint by
        Proposition~~\ref{prop:canoe_distinct}. But the edge subpath
        $p_n, q_n, p_1$ indicates $g_1N_{k-1} \cap g_nN_{k-1} \neq
        \emptyset$. Thus, $S_k$ is a tree.
   \end{proof}
   
   \noindent {\bf Conclusion.} We now use property (I2) to conclude
   the proof of Theorem~\ref{thm:new_CMM}. That is, we define a subset
   $\cO \subset V\cP$ so that $\la G_c
   \ra_{c \in V\cP} \leq G$ is isomorphic to the free product
   $\Asterisk_{c \in \cO} G_c$. First we check that $G_k \cong G_{k-1}
   * \bigl( *_{v \in \cO_{k}} G_v \bigr)$ for each $k\geq 1$. The group $G_k$ acts on $W_k$ preserving the covering by the
        translates of $N_{k-1}$ and so it acts on the skeleton
        $S_k$. The 
        edge stabilizers are trivial by
        Claim~\ref{claim:vertex_intersect}. There is one $G_k$-orbit
        in the vertex set $V_1$, and the group $G_{k-1}$ stabilizes
        the vertex corresponding to $N_{k-1}$. Therefore, the free
        product decomposition of $G_{k}$ follows from the definition of
        $\cO_{k}$ and Bass--Serre theory.
        The quotient $S_k/G_k$ is
        also a tree with a vertex representing $V_1$ and vertices
        representing orbits in $V_2$, all connected to $V_1$.
   
       \begin{figure}[h]
         \centering
         \def\svgwidth{.85\columnwidth}
      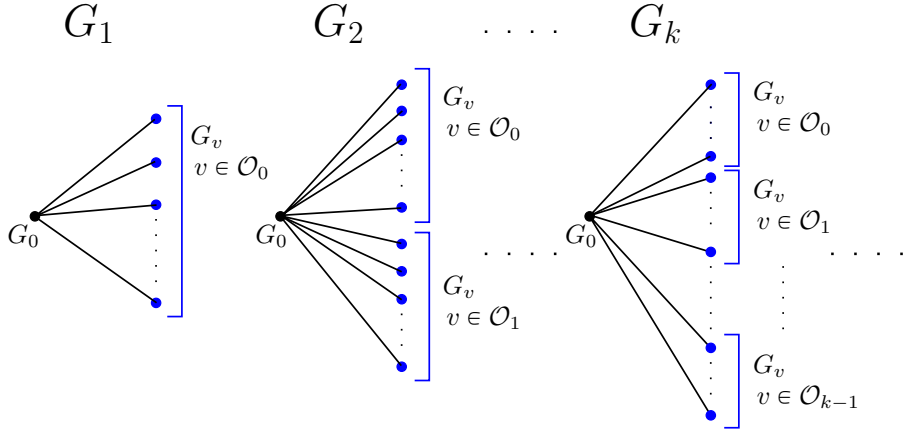
      \caption{Directed system of graphs of groups decompositions for the groups $\{G_{k}\}$.}
    \label{fig:bassserreinduction}
   \end{figure}
   
   Since the
   windmills exhaust the projection complex, $\la G_c \ra_{c \in V\cP}
   = \varinjlim_{k} G_k$. Finally, $\varinjlim_{k} G_k = \Asterisk_{c
     \in \cO} G_c$ for $\cO = \cup_{k=-1}^{\infty} \cO_k$, which again can
   be deduced from a Bass-Serre theory argument as follows.
   
   We will specify an increasing union of trees so that the group $\varinjlim_{k}G_{k}$ acts on the direct limit tree. 
   Recall that (I2) yields for each $k$ a graph of groups decomposition of $G_{k}$ with vertex groups $G_{k-1}$ and $G_v$ for each $v \in \cO_{k}$. There is an edge $\{G_v, G_{k-1}\}$ with trivial edge group for each $v \in \cO_{k}$. As depicted in Figure~\ref{fig:bassserreinduction}, the graph of groups decomposition for~$G_2$ can be expanded using the graph of groups decomposition for $G_1$. More specifically, in the graph of groups decomposition for $G_2$, delete the vertex for $G_1$, and replace it with the graph of groups decomposition for $G_1$, attaching every group $G_v$ for $v \in \cO_2$ to the vertex $G_0$ with trivial edge group. The group $G_2$ then acts on the new corresponding Bass-Serre tree. Continue this recursive procedure: in the graph of groups decomposition for $G_k$, delete the vertex for $G_{k-1}$ and replace it with the recursively obtained graph of groups decomposition for $G_{k-1}$, attaching every group $G_v$ for $v \in \cO_{k}$ to $G_0$ with
  trivial edge group. This process yields an increasing union of Bass--Serre trees, and the $\varinjlim_{k}G_{k}$ acts on the direct limit tree as desired.
   \end{proof}
   
   \section{Free products from rotating families} \label{sec:DGO_proof}

    The aim of this section is to combine Theorem~\ref{thm:proj_com_from_DGO} and Theorem~\ref{thm:new_CMM} to give a new proof of the following theorem of Dahmani--Guirardel--Osin with different constants. 

    \begin{thm}  \label{thm:new_proof_DGO}
        Let $G$ be a group acting by isometries on a $\delta$-hyperbolic metric space with $\delta \geq 1$, and let $\cC = (C, \{G_c \, | \, c \in C\})$ be a $\rho$-separated fairly rotating family for some $\rho \ge 2\delta \log_{2}(\delta)+ 38 \delta$. Then, the normal closure in $G$ of the set $\{G_c\}_{c \in C}$ is isomorphic to a free product $\Asterisk_{c \in C'} G_c$, for some (usually infinite) subset $C' \subset C$.
    \end{thm}
    \begin{proof}
      Take $\theta = 121\delta$, $K = 3\theta$, and let $R = \delta\log_2(\delta)+16\delta$, which meets the constraint $2+2\delta \le R \le \frac{\rho}{2} - 3\delta$. Then by
      Theorem~\ref{thm:proj_com_from_DGO}, the group $G$ acts by isometries on a projection
      complex $\cP = \cP(C,\theta, K)$ obtained from a collection
      $(C, \{d_p\}_{p\in C})$ satisfying the strong projection axioms, and the family of subgroups $\{G_c\}$ is an equivariant $L$-spinning family for $L = 2^{\frac{R-2}{\delta}} - 4-248\delta$.
        
        
        One can check that our choice of $R$ satisfies $L>4M+K$, where $M$ is the Bounded Geodesic Image Theorem constant given in Theorem~\ref{thm:bgi}. Indeed, as $R = \delta\log_2(\delta)+16\delta$, we have the following equivalent inequalities:
        \begin{align*}
         L &> 4M+K, \\
         2^{\frac{R-2}{\delta}}- 4-248\delta &> 4(8K+2\theta)+K, \\
        2^{\frac{R-2}{\delta}} &> 13195\delta + 4.
        \end{align*}
        Since $\delta \geq 1$ it suffices to check
        \begin{align*}
        65536\delta=2^{\frac{R}{\delta}} &> 4(13199\delta) = 52796\delta.
        \end{align*}       
        Thus, the hypotheses of Theorem~\ref{thm:new_CMM} are satisfied, so $\la \la G_c \ra\ra_{c \in C} \leq G$ is isomorphic to a free product $\Asterisk_{c \in C'} G_c$, for some subset $C' \subset C$ as desired. 
    \end{proof}
    
\section{Loxodromic Elements}

In this final section we prove the second halves of Theorems \ref{DGO} and \ref{thm:cmm_intro} which state that our subgroup of $G$ consists of elements that are either point stabilizers in some $G_{c}$ or act loxodromically on both the hyperbolic metric space $X$ and the projection complex $\cP$. We begin with the action on the projection complex.

\begin{prop} \label{prop:projloxodromic}
	Let $\cP$, $G$, and $\{G_{c}\}_{c \in V\cP}$ be as in Theorem \ref{thm:new_CMM}. Then every element of the subgroup of $G$ generated by $\{G_{c}\}_{c \in V\cP}$ is either loxodromic in $\cP$ or is contained in some $G_{c}$. 
\end{prop}

\begin{proof}
	Let $g$ be an element of the group generated by $\{G_{c}\}_{c
          \in C\cP}$. By the proof of Theorem \ref{thm:new_CMM}, $g$ is
        contained in $G_{k}$ for some $k$. Now $G_{k}$ acts on the
        Bass-Serre tree which is the skeleton, $S_{k}$, of the cover
        of $W_{k}$ by the translates of $N_{k-1}$. Let us first assume
        that $g$ acts on this tree loxodromically. Let $x_{0}$, an
        intersection point of two translates of $N_{k-1}$, be a point on the axis of $g$ in $S_{k}$. Thus in $S_{k}$ we have that $d_{S_{k}}(x_{0},g^{n}x_{0})$ grows linearly in $n$. 
	
	Now we move from the Bass-Serre tree back to $\cP$. Note that
        $x_{0}$ is also a point in $\cP$ and consider the orbit of
        $x_{0}$ in $\cP$. Also, $x_{0}$ and all of its translates are
        large angle intersection points of distinct translates of
        $N_{k-1}$. Given any $n$ we can apply (I1) to form an
        $L$-canoeing path $\gamma$ from $x_{0}$ to $g^{n}x_{0}$. Let
        $m=d_{S_{k}}(x_{0},g^{n}x_{0})$. Since each of the
        $g^{i}x_{0}$ are intersection points between translates of the
        $N_{k-1}$, we can apply the furthermore statement of (I2) to
        see that the number of large angle points on $\gamma$ is at
        least $\frac{m}{2}-1$. Apply Corollary \ref{cor:cpbound} to
        see that $d(x_{0},g^{n}x_{0}) \geq \frac{m-2}{4}$ with $m$
        growing linearly in $n$. We conclude that the translation
        length of $g$ is strictly positive and hence $g$ acts
        loxodromically on $\mathcal P$.
	
	Now if $g$ fixes a point in $S_{k}$ then it is conjugate into either one of the $G_{c}$ or $G_{k-1}$. However, now we can just run the argument again in $G_{k-1}$, continuing until $G_{0} = G_{v_{0}}$ if necessary. 
\end{proof}

We next see that we can push this result forward again to the original $\delta$-hyperbolic space, $X$. 

\begin{prop} \label{prop:metricloxodromic}
	Let $G$ and $\cC$ be as in Theorem \ref{thm:new_proof_DGO}. Then every element of the subgroup of $G$ generated by the set $\{G_{c}\}_{c\in C}$ is either a loxodromic isometry of $X$ or it is contained in some $G_{c}$. 
\end{prop}

\begin{proof}
  We first apply Theorem \ref{thm:new_proof_DGO} and run the argument above in $\cP$. Thus for any $g\in G$ we either have $g \in G_{c}$ for some $c$ or we have an orbit $\{g^{n}x_{0}\}$ such that for any $n\geq 2$ we have that $d_{g^{i}x_{0}}(x_{0},g^{n}x_{0}) > K > 4\delta$ for all $i = 1,\ldots,n-1$. Thus by Lemma \ref{boundedproj} we have that every geodesic from $x_{0}$ to $g^{n}x_{0}$ passes through each of the balls $B_{R+2\delta}(g^{i}x_{0})$ for $i=1,\ldots,n-1$. Now our choice of $\rho$ and $R$ guarantees that each of these balls are distance at least $2\delta$ from each other so that $d(x_{0},g^{n}x_{0}) \geq 2\delta(n-1)$. We conclude that the translation length of $g$ is strictly positive and hence $g$ is loxodromic.
\end{proof}

\bibliographystyle{alpha}
\bibliography{Ref}

\end{document}